\documentclass{amsproc}

\usepackage[utf8]{inputenc}
\usepackage[english]{babel}
\usepackage{comment}
\usepackage[alphabetic]{amsrefs}
\usepackage{tikz}
\usepackage{xcolor}
\usepackage{datetime2} 
\usepackage[colorlinks=true,linkcolor=blue,]{hyperref}

\usepackage{stackengine}
\usepackage{scalerel}
\usepackage{mathtools}
\usepackage{calc}
\usepackage{amsthm}
\usepackage{thmtools}
\usepackage[framemethod=TikZ]{mdframed}
\usepackage{amssymb}
\usepackage{amsfonts}
\usepackage{euscript}
\usepackage{fourier}
\def\dotminus{\stackon[.2ex]{$-$}{$.$}}
\usepackage{dsfont}
\renewenvironment{itemize}
  {\begin{list}{$\triangleright$}{%
  \setlength{\parskip}{0mm}
  \setlength{\topsep}{.4\baselineskip}
  \setlength{\rightmargin}{0mm}
  \setlength{\listparindent}{0mm}
  \setlength{\itemindent}{0mm}
  \setlength{\labelwidth}{3ex}
  \setlength{\itemsep}{.2\baselineskip}
  \setlength{\parsep}{.2\baselineskip}
  \setlength{\partopsep}{0mm}
  \setlength{\labelsep}{1ex}
  \setlength{\leftmargin}{\labelwidth+\labelsep}
  }}{%
\end{list}}

\declaretheoremstyle[
  headfont=\normalfont\bfseries,
  notefont=\bfseries,
  notebraces={(}{)},
  bodyfont=\normalfont,
  postheadspace=1em,
  mdframed={
  nobreak=true,
  outerlinewidth=1pt,
  linecolor=gray!20,
  roundcorner = 1ex,    
  backgroundcolor=gray!10, 
  innerleftmargin=1ex,
  leftmargin=0ex,
  innerrightmargin=1ex,
  rightmargin=0ex,
  innertopmargin=1.5ex, 
  innerbottommargin=1ex, 
  skipabove=3ex,
  skipbelow=1ex}, 
]{mystyle}

\declaretheorem[style=mystyle]%
{theorem}
\declaretheorem[style=mystyle,sibling=theorem]%
{lemma,proposition,fact,corollary,definition,notation,remark,example,claim,question}

\declaretheoremstyle[
  spaceabove=6pt, 
  spacebelow=6pt, 
  headfont=\normalfont\itshape, 
  bodyfont = \normalfont,
  postheadspace=1em, 
  qed=\qedsymbol, 
  headpunct={.}]
{myproof} 
\declaretheorem[style=myproof, unnumbered]{proof}

\renewcommand*{\emph}[1]{%
   \smash{\tikz[baseline]\node[rectangle, fill=teal!25, rounded corners, inner xsep=0.5ex, inner ysep=0.2ex, anchor=base, minimum height = 2.7ex]{\strut #1};}}

\author{Domenico Zambella}
\thanks{Dipartimento di Matematica, Universit\`a di Torino, via Carlo Alberto 10, 10123 Torino.}
\begin{document}
\title{Standard analysis}
\hfill\texttt{~}
\maketitle
\raggedbottom

\begin{abstract}
  Let ${\EuScript L}$ be a first-order two-sorted language.
  Let $S$ be some fixed structure.
  A \textit{standard\/} structure is an ${\EuScript L}$-structure of the form $\langle M,S\rangle$, where $M$ is arbitrary.
  When $S$ is a compact topological space (and ${\EuScript L}$ meets a few additional requirements) it is possible to adapt a significant part of model theory to the restricted class of standard structures.
  This has been shown by Henson and Iovino for normed spaces (see, e.g.~\cite{HI} and references cited therein) and has been generalized to a larger class of structures in~\cite{clcl}.
  We further generalize their approach to a significantly larger class.\\[1ex]
  \noindent
  The starting point is to prove that every standard structure has a \textit{positive\/} elementary extension that is standard and realizes all positive types that are finitely consistent.
  A second important step is to prove that (in a sufficiently saturated structure) the negation of a positive formula is equivalent to an infinite (but small) disjunction of positive formulas.
  The main tool is the notion of approximation of a positive formula and of its negation that have been introduced by Henson and Iovino.\\[1ex]
  \noindent
  We review and elaborate on the properties of positive formulas and their approximations.
  In parallel, we introduce \textit{continuous\/} formulas which provide a better counterpart to Henson and Iovino theory of normed spaces and the real-valued model theory of metric spaces of~\cite{BBHU}.
  To demonstrate this setting in action we discuss $\omega$-categoricity and (local) stability.
\end{abstract}


\section{Standard structures}\label{uno}

\def\ceq#1#2#3{\parbox[t]{23ex}{$\displaystyle #1$}\parbox{6ex}{\hfil $#2$}{$\displaystyle #3$}}

In order to keep the paper self-contained, in the first few sections we revisit and extend some results of~\cite{clcl}.
Some theorems are rephrased and we add a few remarks and examples.

Let \emph{$S$\/} be some fixed first-order structure which is endowed with a Hausdorff compact topology (in particular, a normal topology).
The language \emph{${\EuScript L}_{\sf S}$\/} contains a relation symbol for each compact subsets $C\subseteq S^n$ and a function symbol for each continuous functions $f:S^n\to S$.
In particular, there is a constant for each element of $S$.
According to the context, $C$ and $f$ denote either the symbols of ${\EuScript L}_{\sf S}$ or their interpretation in the structure $S$.
Such a language ${\EuScript L}_{\sf S}$ is much larger than necessary but it is convenient because it uniquely associates a structure to a topological space.
The notion of a \textit{dense set of formulas\/} (Definition~\ref{def_dense}) helps to reduce the size of the language when required.

The most straightforward examples of structures $S$ are the unit interval $[0,1]$ with the usual topology, and its homeomophic copies $\mathds{R}^+\cup\{0,\infty\}$ and $\mathds{R}\cup\{\pm\infty\}$.
Interesting examples are obtained when $S=S_n(M)$, the space of complete $n$-types over a model $M$ with the logic topology.
These will be discussed in subsequent work.

We also fix a first-order language \emph{${\EuScript L}_{\sf H}$\/} which we call the language of the \emph{home sort.}

\begin{definition}\label{def_0}
  Let \emph{${\EuScript L}$\/} be a two sorted language. 
  The two sorts are denoted by \emph{${\sf H}$} and \emph{${\sf S}$.} 
  The language ${\EuScript L}$ expands both ${\EuScript L}_{\sf H}$ and ${\EuScript L}_{\sf S}$.
  It also has some symbols sort ${\sf H}^n\times{\sf S}^m\to {\sf S}$.
  An \emph{${\EuScript L}$-structure\/} is a structure of signature ${\EuScript L}$ that interprets these symbols in equicontinuous functions (definition below).

  A \emph{standard structure\/} is a two-sorted ${\EuScript L}$-structure of the form $\langle M,S\rangle$, where $M$ is any structure of signature ${\EuScript L}_{\sf H}$ and $S$ is fixed.
  
  For convenience we require that ${\EuScript L}$ also has a  relation symbol $r_{\varphi}(x)$ for every $\varphi(x)\in {\EuScript L}_{\sf H}$.
  All ${\EuScript L}$-structures are assumed to model $r_{\varphi}(x)\leftrightarrow\varphi(x)$.
  In other words, the Morleyzation of ${\EuScript L}_{\sf H}$ is assumed in the definition of ${\EuScript L}$-structure.

  Standard structures are denoted by the domain of their home sort.
\end{definition}


Let $f:A\times S^m\to S$, where $A$ is a pure set.
Let $\varepsilon$ and $\delta$ range over the closed neighborhoods of the diagonal of $S\times S$, respectively over the open neighborhoods of the diagonal of $S^m\times S^m$.
We say that $f$ is \emph{equicontinuous\/} if for every $\varepsilon$ there is a $\delta$ such that $\big\langle f(a,\alpha),f(a,\alpha')\big\rangle\in\varepsilon$ for every $a\in A$ and every $\langle\alpha,\alpha'\rangle\in\delta$.

A \emph{modulus of equicontinuity\/} for $f$ is a function that maps $\varepsilon$ to some $\delta$ that satisfies the condition above.

Clearly, saturated ${\EuScript L}$-structures exist but, except for trivial cases (when $S$ is finite), these are not standard.
As a remedy, below we carve out a set of formulas ${\EuScript F}^{\rm p}$, the set of positive formulas, such that every model has a positive elementary saturated extension that is also standard.

As usual, ${\EuScript L}_{\sf S}$, ${\EuScript L}_{\sf H}$, and ${\EuScript L}$ denote both first-order languages and the corresponding set of formulas.
We write ${\EuScript L}_x$ when we restrict variables to $x$.
Up to equivalence we may assume that ${\EuScript L}$ has two types of atomic formulas

\begin{itemize}
  \item[1.] $r_\varphi(x)$ for some $\varphi(x)\in{\EuScript L}_{\sf H}$;
  \item[2.] $\tau(x\,;\eta)\in C$, where $C\subseteq S^n$ is a compact set and $\tau(x\,;\eta)$ is a tuple of terms of sort ${\sf H}^{|x|}\times {\sf S}^{|\eta|}\to {\sf S}$ (up to equivalence, equality has also this form);
\end{itemize}

\begin{definition}\label{def_LL}
  A formula in ${\EuScript L}$ is \emph{positive\/} if it uses only the Boolean connectives $\wedge$, $\vee$; the quantifiers $\forall\raisebox{1.1ex}{\scaleto{\sf H}{.8ex}\kern-.2ex}$, $\exists\raisebox{1.1ex}{\scaleto{\sf H}{.8ex}\kern-.2ex}$ of sort ${\sf H}$; and the quantifiers $\forall\raisebox{1.1ex}{\scaleto{\sf S}{.8ex}\kern-.2ex}$, $\exists\raisebox{1.1ex}{\scaleto{\sf S}{.8ex}\kern-.2ex}$ of sort ${\sf S}$.
  The set of positive formulas is denoted by \emph{${\EuScript F}^{\rm p}$.}

  A \emph{continuous\/} formula is a positive formula where only atomic formulas as in (2) above occur.
  The set of continuous formulas is denoted by \emph{${\EuScript F}^{\rm c}$.}


  We will use Latin letters $x,y,z$ for variables of sort ${\sf H}$ and Greek letters $\eta,\varepsilon$ for variables of sort ${\sf S}$.
  Therefore we can safely drop the superscript from quantifiers if they are followed by variables.
\end{definition}


\begin{notation}
  Much of the theory below is developed in parallel in the positive and the continuous case.
  For the sake of conciseness, we write \emph{p/c} as a placeholder that can be replaced for either one of \emph{p} or \emph{c.}
\end{notation}

The positive formulas of Henson and Iovino correspond to those formulas in ${\EuScript F}^{\rm c}$ that do not have quantifiers of sort ${\sf S}$, nor functions symbols of sort ${\sf H}^n\times{\sf S}^m\to {\sf S}$ with $nm>0$.

Allowing quantifiers of sort ${\sf S}$ comes almost for free.
In Section~\ref{cIelimination} we will see that the formulas in ${\EuScript F}^{\rm p/c}$ are approximated by formulas in ${\EuScript F}^{\rm p/c}$ without quantifiers of sort ${\sf S}$.

Instead, choosing ${\EuScript F}^{\rm p}$ over ${\EuScript F}^{\rm c}$ comes at a price.
However ${\EuScript F}^{\rm p}$ simplifies the comparision with classical model theory.

An important fact to note about ${\EuScript F}^{\rm c}$ is that it is a language without equality of sort ${\sf H}$.
The difference between ${\EuScript F}^{\rm p}$ and ${\EuScript F}^{\rm c}$ is similar to the difference between the space of absolutely integrable functions (w.r.t.\@ some measure $\mu$) and the Lebesgue space $L^1(\mu)$.
For the latter only equality almost everywhere is relevant.
However, even when $L^1(\mu)$ is our focus of interest, it is often easier to argue about real valued functions.
For a similar reason ${\EuScript F}^{\rm p}$ and ${\EuScript F}^{\rm c}$ are better studied in parallel.

The extra expressive power offered by quantifiers of sort ${\sf S}$ is convenient.
For instance, the example below should convince the reader that ${\EuScript F}^{\rm c}$ has at least the same expressive power as real valued logic, see~\cite{BBHU} or~\cite{K}.

\begin{example}\label{ex_Rvlogic}
  Let $S=[0,1]$, the unit interval.
  Let $\tau(x)$ be a term of sort ${\sf H}^{|x|}\to {\sf S}$.
  Then there is a positive formula that says $\sup_x \tau(x)=\alpha$.
  Indeed, the expression
  
  \ceq{\hfill\forall x\ \big[\tau(x)\le\alpha\big]}{\wedge}{\forall \varepsilon>0\ \exists x\ \big[\alpha\le \tau(x)+ \varepsilon\big].}
  
  is formalized by the ${\EuScript F}^{\rm c}$-formula  

  \ceq{\hfill\forall x\ \big[\tau(x)\dotminus\alpha\in\{0\}\big]}
  {\wedge}{\forall \varepsilon \Big[\varepsilon\in\{0\}\ \vee\ \exists x\ \big[(\alpha\dotminus \tau(x))\dotminus\varepsilon\in\{0\}\big]\Big].}
\end{example}

A precise comparison with real-valued logic is not straightforward and outside the scope of this paper.
In real-valued logic $S=[0,1]$ is interpreted as a set of truth values.
Functions of sort ${\sf H}^n\to {\sf S}$ are called atomic formulas; functions of sort ${\sf S}^m\to {\sf S}$ are called propositional connectives.
Real-valued logic does not provide any counterpart for functions sort ${\sf H}^n\times{\sf S}^m\to {\sf S}$ when $nm>0$.

\hfil***

Normed spaces are among the motivating examples for the study of standard structures.
The requirement for $S$ to be compact may appear as a limitation, but it is not a serious one.
In fact, functional analysis and model theory are only concerned with the unit ball of normed spaces.

To formalize the unit ball of a normed space as a standard structure (where $S$ is the unit interval) is straightforward in theory but cumbersome in practice.
Alternatively, one could view a normed space as a many-sorted structure: a sort for each ball of radius $n\in\mathds{Z}^+$.
Unfortunately, this also results in a bloated formalism.

A neater formalization is possible if one accepts that some positive elementary extensions (to be defined below) of a normed space could contain vectors of infinite norm hence not be themselves proper normed spaces.
However these infinities are harmless if one restricts to balls of fine radius.
In fact, inside any ball of finite radius these improper normed spaces are completely standard.

\begin{example}\label{ex_normed_spaces}
Let $S=\mathds{R}\cup\{\pm\infty\}$.
Let ${\EuScript L}_{\sf H}$ be the language of real (or complex) vector spaces.
The language ${\EuScript L}$ contains a function symbol $\|\mbox{-}\|$ of sort ${\sf H}\to {\sf S}$.
Normed spaces are standard structures with the natural interpretation of the language.
It is easy to verify that the unit ball of a standard structure is the unit ball of a normed space.
Note that the same remains true if we add to the language any continuous (equivalently, bounded) operator or functional.
\end{example}

The following example is particularly interesting and will be discussed in detail in subsequent work.

\begin{example}
  Let $G,S$ be an equicontinuous $G$-flow as defined in~\cite{A}.
  Let ${\EuScript L}_{\sf H}$ be the language of groups with possibly some additional symbols.
  The only relevant function in ${\EuScript L}$ is the group action.
  Then $\langle G,S\rangle$ is a standard structure.
\end{example}

\section{Compactness}\label{compactness}

\def\ceq#1#2#3{\parbox[t]{18ex}{$\displaystyle #1$}\parbox[t]{6ex}{\hfil $#2$}{$\displaystyle #3$}}

In this section we prove the compactness theorem for positive theories. 
But first we introduce the notion of standard structure.
This allows to derive the positive case from its classical counterpart.
We will be sketchy and refer the reader to~\cite{clcl} for details (though we work with a larger ${\EuScript L}$, the argument is similar).

We recall the notion of standard part of an element of the elementary extension of a compact Hausdorff topological space.
Our goal is to prove Lemma~\ref{lem_st} which in turn is required for the proof of the positive compactness theorem (Theorem~\ref{thm_compattezza}).
The reader willing to accept the latter without proof may skip this section.

Let $\langle N, {}^*\!\!S\rangle$ be an ${\EuScript L}$-structure. 
Let $\eta$ be a free variable of sort ${\sf S}$.
For each $\beta\in S$, we define the type

\ceq{\hfill{\rm m}_\beta(\eta)}{=}{\{\eta\in D\ :\  D \textrm{ compact neighborhood of }\beta\}.}

In nonstandard analysis, the set of the realizations of ${\rm m}_\beta(\eta)$ in ${}^*\!\!S$ is called the monad of $\beta$. 

The compactness of $S$ ensures that for every $\alpha\in{}^*\!\!S$ there is a unique $\beta\in S$ such that ${}^*\!\!S\models{\rm m}_\beta(\alpha)$.
We denote by \emph{${\rm st}(\alpha)$\/} the unique $\beta\in S$ such that ${}^*\!\!S\models{\rm m}_\beta(\alpha)$.
The following fact is clear.

\begin{fact}\label{fact_st1}
  For every $\alpha\in{}^*\! S$ and every compact $C\subseteq S$

\ceq{\hfill{}^*\!\!S}{\models}{\alpha\in C\ \rightarrow\ {\rm st}(\alpha)\in C.}
\end{fact}


We also need the following.

\begin{lemma}\label{fact_terms_st}
  For every $a\in N^n$, $\alpha\in({}^*\!\!S)^m$ and every term $\tau$ of the right sort

  \ceq{\hfill\langle N,{}^*\!\!S\rangle}{\models}{{\rm st}\big(\tau(a,\alpha)\big)={\rm st}\big(\tau\big(a,{\rm st}(\alpha)\big)\big).}
\end{lemma}

\begin{proof}
  We prove the lemma for function symbols -- the general case follows easily by induction on the syntax of $\tau$.
  Let $\varepsilon$ and $\delta$ below be those given by some modulus of uniform continuity of $\langle N,{}^*\!\!S\rangle$.
  We know that  
  
  \ceq{\hfill \langle N,{}^*\!\!S\rangle}{\models}{\forall x,\eta,\eta'\big[\langle\eta,\eta'\rangle\in\delta\rightarrow \big\langle f(x,\eta),f(x,\eta')\big\rangle\in\varepsilon\big].}
  
  As $\langle N,{}^*\!\!S\rangle\models\big\langle\alpha,{\rm st}(\alpha)\big\rangle\in\delta$ for every $\delta$, we conclude that $\big\langle f(a,\alpha),f\big(a,{\rm st}(\alpha)\big)\big\rangle\in\varepsilon$ for every $\varepsilon$.
  Therefore $f(a,\alpha)$ and $f\big(a,{\rm st}(\alpha)\big)$ have the same standard part.
\end{proof}

The \emph{standard part of $\langle N,{}^*\!\!S\rangle$\/} is the standard structure $\langle N,S\rangle$ that interprets the symbols $f$ of sort ${\sf H}^n\times{\sf S}^m\to {\sf S}$ as the functions

\ceq{\hfill f^N(a,\alpha)}{=}{{\rm st}\big({}^*\!\!f(a,\alpha)\big)}\hfill for all $a\in N^n$, $\alpha\in S^{|\alpha|}$.

where ${}^*\!\!f$ is the interpretation of $f$ in $\langle N,{}^*\!\!S\rangle$.
By the fact above, when $n=0$ the definition agrees with the canonical interpretation of the functions in ${\EuScript L}_{\sf S}$.
Symbols in ${\EuScript L}_{\sf H}$ maintain the same interpretation.
Note that the formulas asserting that the functions ${}^*\!\!f(a,\mbox{-})$ are equicontinuous are in ${\EuScript F}^{\rm p}$.
Therefore the functions ${}^N\!f(a,\mbox{-})$ are also equicontinuous by Lemma~\ref{lem_st} below.

Easy induction proves the following.

\begin{fact}\label{fact_st2} 
  With the notation as above.
  For every $a\in N^n$ and $\alpha\in ({}^*\!\!S)^m$

  \ceq{\hfill \tau^N\big(a\,;{\rm st}(\alpha)\big)}{=}{{\rm st}\big({}^*\!\tau(a\,;\alpha)\big),}

  where $\tau^N$ and ${}^*\!\tau$ are interpretations of $\tau$ in $N$, respectively $\langle N,{}^*\!\!S\rangle$.
\end{fact}

Now we move to positive formulas.


\begin{lemma}\label{lem_st}
  With the notation as above.
  Let $\varphi(x\,;\eta)\in{\EuScript F}^{\rm p}$, \  $a\in N^{|x|}$, and $\alpha\in({}^*\!\!S)^{|\eta|}$ be arbitrary then $\langle N,{}^*\!\!S\rangle\models\varphi(a\,;\alpha)$ implies that $N\models\varphi\big(a\,;{\rm st}(\alpha)\big)$
\end{lemma}

\begin{proof}
  Suppose $\varphi(x\,;\eta)$ is ${\EuScript F}^{\rm p}$-atomic.
  If $\varphi(x\,;\eta)$ is a formula of ${\EuScript L}_{\sf H}$ the claim is trivial. 
  Otherwise, assume that $\varphi(x\,;\eta)$ has the form $\tau(x\,;\eta)\in C$.
  Assume that the tuple $\tau(x\,;\eta)$ consists of a single term.
  The general case follows easily from this special case. 
  Assume that $\langle N,{}^*\!\!S\rangle\models \tau(a\,;\alpha)\in C$.
  Then ${\rm st}\big({}^*\tau(a\,;\alpha)\big)\in C$ by Fact~\ref{fact_st1}.
  Therefore $\tau^N\big(a\,;{\rm st}(\alpha)\big)\in C$ follows from Fact~\ref{fact_st2}.
  This proves the lemma for atomic formulas.
  Induction is immediate. 
\end{proof}

\begin{corollary}\label{corol_st}
  Let $M$ be a standard structure.
  Let $p(\eta)\subseteq{\EuScript F}^{\rm p}(M)$ be a type that does not contain existential quantifiers of sort ${\sf H}$.
  Then, if $p(\eta)$ is finitely consistent in $M$, it is realized in $M$.
 \end{corollary}
 
 The corollary has also a direct proof.
 This goes through the observation that the formulas in $p(\eta)$ define compact subsets of $S$.

 \begin{proof}
  Let $\langle ^*\!M,^*\!\!S\rangle$ be an ${\EuScript L}$-elementary saturated superstructure of $\langle M,S\rangle$.
  Pick some $\alpha\in{}^*\!\!S$ such that $\langle ^*\!M,^*\!\!S\rangle\models p(\alpha)$.
  By Lemma~\ref{lem_st}, $^*\!M\models p({\rm st}(\alpha))$.
  Now observe that the truth of formulas without the quantifier $\exists\raisebox{1.1ex}{\scaleto{\sf H}{.8ex}}$ are preserved by substructures.
 \end{proof}
 
 The exclusion of existential quantifiers of sort ${\sf H}$ is necessary.
 For a counterexample take $M=S=[0,1]$.
 Assume that ${\EuScript L}$ contains a function symbol for the identity map $\iota:M\to S$.
 Let $p(\eta)$ contain the formulas $\exists x\,\big(x\neq0\,\wedge\;\iota x+\eta\in[0,1/n]\big)$ for all positive integers $n$.



\hfil***

We are ready for the proof of the positive compactness theorem.
It is convenient to distinguish between consistency with respect to standard structures and consistency with respect to ${\EuScript L}$-structures.
We say that a theory $T$ is \emph{${\EuScript L}$-consistent\/} if $\langle M,{}^*\!\!S\rangle\models T$ for some ${\EuScript L}$-structure $\langle M,{}^*\!\!S\rangle$ (this is the classical notion of consistency).
We say that $T$ is \emph{uniformly finitely\/} ${\EuScript L}$-consistent if every finite $T_0\subseteq T$ is ${\EuScript L}$-consistent and the ${\EuScript L}$-structures witnessing this have a common modulus of equicontinuity.

Finally, we say that $T$ is \emph{standardly\/} consistent if $M\models T$ for some standard structure $M$.
By Lemma~\ref{lem_st} these two notions of consistency coincide when $T$ is positive.
Therefore we have the following.

\begin{theorem}[positive compactness theorem]\label{thm_compattezza}
  Let $T$ be a positive theory. Then, if $T$ is uniformly finitely ${\EuScript L}$-consistent, it is also standardly consistent.
\end{theorem}

\begin{proof}
  By uniform finite consistency, we can assume without loss of generality that $T$ contains some positive sentences asserting that the interpretations of the function of sort ${\sf H}^n\times{\sf S}^m\to{\sf S}$ are equicontinuous.
  Therefore, by the classical compactness theorem, $T$ is ${\EuScript L}$-consistent.
  Then the theorem follows from Lemma~\ref{lem_st}
\end{proof}





An ${\EuScript L}$-structure $N$ is \emph{positively $\lambda$-saturated\/} if it realizes all types $p(x\,;\eta)\subseteq{\EuScript F}^{\rm p}(N)$ that are finitely consistent in $N$ and have fewer than $\lambda$ parameters.
When $\lambda=|N|$ we simply say \emph{p-saturated.}

We say that $M$ is a \emph{p/c-elementary\/} substructure of $N$ if $M\subseteq N$ and

\ceq{\hfill M\models\varphi(a)}{\Rightarrow}{N\models\varphi(a)}

for every $\varphi(x)\in{\EuScript F}^{\rm p}$ and every $a\in M^{|x|}$.

We say that $M$ is \emph{p/c-maximal\/} if it models all sentences in ${\EuScript F}^{\rm p/c}(M)$ that holds in some p-elementary extension of $M$.

\begin{theorem}
  Every standard structure has a p-elementary extension to a p-satu\-rated p-maximal standard structure.
\end{theorem}

\begin{proof}
  Up to a small tweak, this is just the classical construction.
  First note that the moduli of equicontinuity for a standard structure $M$ are also moduli of equicontinuity for any of its p-elementary extensions.
  Now, suppose that every finite subset of $p(x,\eta)\subseteq{\EuScript F}^{\rm p}(M)$ is realized in some p-elementary extension of $M$.
  Then, by the positive compactness theorem, $p(x,\eta)$ is realized in some p-elementary extension of $M$.

  With this in mind we construct a p-elementary chain of length $\lambda$ of standard structures $M_i$ of cardinality $\le\lambda$ and a sequence of types $p_i(x)\subseteq{\EuScript F}^{\rm p}(M_i)$ with $<\lambda$ parameters that are realized in $M_{i+1}$.
  If $\lambda=\lambda^{<\lambda}$ we can ensure that all types finitely consistent in some p-elementary superstructure of $M_i$ occur in the enumeration.
  Let $N$ be the union of the chain.
  If $p(x)\subseteq{\EuScript F}^{\rm p}(N)$ is finitely consistent in $N$ and has $<\lambda$ parameters then, by regularity, it occurs at some stage $i<\lambda$. 
  Then $p(x)$ is realized in $M_{i+1}$ and by p-elementarity, in $N$.
  The claim of p-maximality is immediate.
\end{proof}



\section{Henson-Iovino approximations}\label{ultrapws} 
The notion of approximation is a central tool in the work of Henson and Iovino on the model theory of normed spaces, see e.g.~\cite{HI}.
In~\cite{clcl} the natural generalization to standard structures is presented.
We recall the definitions and the main properties.

For $C,C'$ compact subsets of $S^n$, we write $C'>C$ if $C'$ is a neighborhood of $C$.
For $\varphi,\varphi'$ (free variables are hidden) positive formulas possibly with parameters we write \emph{$\varphi'>\varphi$\/} if $\varphi'$ is obtained by replacing in $\varphi$ each atomic formula of the form $t\in C$ with $t\in C'$, for some $C'>C$.
If no such atomic formulas occurs in $\varphi$, then $\varphi>\varphi$.
We call $\varphi'$ a \emph{weakening\/} of $\varphi$.
Note that $>$ is a dense transitive relation and that  $\varphi\to\varphi'$ in every ${\EuScript L}$-structure.
For every type $p(x)$, we write 

\ceq{\hfill \emph{$p'(x)$}}{=}{\big\{\varphi'(x)\ :\ \varphi'>\varphi\textrm{ for some }\varphi(x)\in p\big\}}

in particular \emph{$\{\varphi(x)\}'$} = $\big\{\varphi'(x)\ :\ \varphi'>\varphi\big\}$.

We write \emph{$\tilde\varphi\perp\varphi$\/} when $\tilde\varphi$ is obtained by replacing each atomic formula $t\in C$ occurring in $\varphi$ with $t\in\tilde{C}$ where $\tilde{C}$ is some compact set disjoint from $C$.
The atomic formulas in ${\EuScript L}_{\sf H}$ are replaced with their negation.
Finally each connective is replaced by its dual i.e., $\vee, \wedge, \exists,\forall$ are replaced by $\wedge,\vee,\forall,\exists$, respectively.
We say that  $\tilde\varphi$ is a \emph{strong negation} of $\varphi$.
It is clear that $\tilde\varphi\rightarrow\neg\varphi$ in every ${\EuScript L}$-structure.

Notice that the weakening and the strong negation of a formula in ${\EuScript F}^{\rm c}$ is also in ${\EuScript F}^{\rm c}$.

\begin{lemma}\label{lem_interpolation}
  For all positive formulas $\varphi$
  \begin{itemize}
    \item[1.]for every $\varphi'>\varphi$ there is a formula $\tilde\varphi\perp\varphi$ such that $\varphi\rightarrow\neg \tilde\varphi\rightarrow\varphi'$;
    \item[2.] for every\, $\tilde\varphi\perp\varphi$ there is a formula $\varphi'>\varphi$ such that  $\varphi\rightarrow\varphi'\rightarrow\neg \tilde\varphi$.
  \end{itemize}
\end{lemma}

\begin{proof}
  If $\varphi\in {\EuScript L}_{\sf H}$ the claims are obvious.
  Suppose $\varphi$ is of the form $t\in C$.
  Let $\varphi'$ be $t\in C'$, for some $C'>C$.
  Let $O$ be an open set such that $C\subseteq O\subseteq C'$.
  Then $\tilde\varphi=(t\in S\smallsetminus O)$ is as required by the lemma.
  Suppose instead that $\tilde\varphi$ is of the form $t\in\tilde{C}$ for some compact $\tilde{C}$ disjoint from $C$.
  By the normality of $S$, there is  $C'>C$ disjoint from $\tilde{C}$.
  Then  $\varphi'=\big(t\in C'\big)$ is as required.
  The lemma follows easily by induction.
\end{proof}

\begin{definition}\label{def_dense}
  A set of formulas ${\EuScript H}\subseteq{\EuScript L}$ is \emph{p/c-dense\/} modulo $T$, a theory, if for every positive/continuous $\varphi(x)$ and every $\varphi'>\varphi$, there is a formula $\psi(x)\in{\EuScript H}$ such that $\varphi(x)\rightarrow\psi(x)\rightarrow\varphi'(x)$ holds in every standard structure that models $T$.
\end{definition}

\begin{example}
  By Lemma~\ref{lem_interpolation}, the set of negations of positive/continuous formulas is p/c-dense.
\end{example}

The following fact is easily proved by induction.

\begin{fact}
  Let ${\EuScript H}$ be a set of formulas closed under the connectives $\wedge$, $\vee$, $\forall\raisebox{1.1ex}{\scaleto{\sf H}{.8ex}\kern-.2ex}$, $\exists\raisebox{1.1ex}{\scaleto{\sf H}{.8ex}\kern-.2ex}$, $\forall\raisebox{1.1ex}{\scaleto{\sf S}{.8ex}\kern-.2ex}$, $\exists\raisebox{1.1ex}{\scaleto{\sf S}{.8ex}\kern-.2ex}$.
  Then ${\EuScript H}$ is p/c-dense if the condition in Definition~\ref{def_dense} is satisfied when $\varphi(x)$ atomic.
\end{fact}

Using the fact, it is immediate to verify the following.

\begin{example}\label{ex_prebase}
  Let ${\EuScript C}$ be a prebase of closed subsets of $S$.
  The set of formulas built inductively from atomic formulas of the form $\tau(x,\eta)\in C$ for $C\in{\EuScript C}$ is c-dense.
\end{example}

\begin{example}
  Let $S=[0,1]$.
  The set of formulas built inductively from atomic formulas of the form $\tau=0$ is c-dense.
  In fact, first note that by Example~\ref{ex_prebase} we can restrict to formulas built inductively from atomic formulas of the form $\tau\in[\alpha,\beta]$.
  Then note that $\tau\in[\alpha,\beta]$ is equivalent to $\tau\dotminus\beta\in\{0\}\ \wedge\ \alpha\dotminus\tau\in\{0\}$.
\end{example}

In Section~\ref{cIelimination} we prove that the set of positive/continuous formulas without quantifiers of sort ${\sf S}$ is p/c-dense.

\hfil***

The following proposition shows that a slight amount of saturation tames the positive formulas.

\begin{proposition}\label{prop_approx}
  Let $N$ be a p-$\omega$-saturated standard structure.
  Then $N\models\{\varphi\}'\leftrightarrow\varphi$, for every $\varphi\in{\EuScript F}^{\rm p}(N)$.
\end{proposition}

\begin{proof}
  We prove $\rightarrow$, the nontrivial implication.
  The claim is clear for atomic formulas.
  Induction for conjunction, disjunction and the universal quantifiers is immediate.
 %
 %
%
%
%
%
%
  We consider case of the existential quantifiers of sort ${\sf H}$.
  Assume inductively
  
  \ceq{\textrm{ih.}\hfill\{\varphi(x,z\,;\eta)\}'}
  {\rightarrow}
  {\varphi(x,z\,;\eta)}

  We need to prove

  \ceq{\hfill\{\exists z\,\varphi(x,z\,;\eta)\}'}
  {\rightarrow}
  {\exists z\,\varphi(x,z\,;\eta)}

  From (ih) we have

  \ceq{\hfill\exists z\,\{\varphi(x,z\,;\eta)\}'}
  {\rightarrow}
  {\exists z\,\varphi(x,z\,;\eta)}

  Therefore it suffices to prove

  \ceq{\hfill\{\exists z\,\varphi(x,z\,;\eta)\}'}
  {\rightarrow}
  {\exists z\,\{\varphi(x,z\,;\eta)\}'}

Replace the variables $x\,;\eta$ with parameters, say $a\,;\alpha$, and assume that $N\models\exists z\,\varphi'(a,z,;\alpha)$ for every $\varphi'>\varphi$.
We need to prove the consistency of the type $\{\varphi'(a,z,;\alpha):\varphi'>\varphi\}$.
By saturation, finite consistency suffices.
This is clear if we show that the antecedent is closed under conjunction.
Indeed, it is easy to verify that if $\varphi_1,\varphi_2>\varphi$ then $\varphi_1\wedge\varphi_2>\varphi'$ for some $\varphi'>\varphi$.
In words, the set of approximations of $\varphi$ is a directed set.

For existential quantifiers of sort ${\sf S}$ we argue similarly.
%
%
%
%
%
\end{proof}


\begin{corollary}\label{corol_pmax}
  For every standard structure $M$, the following are equivalent
  \begin{itemize}
    \item[1.] $M$ is p-maximal;
    \item[2.] $M\models\{\varphi\}'\leftrightarrow\varphi$ for every $\varphi\in{\EuScript F}^{\rm p}(M)$ 
  \end{itemize}
\end{corollary}
\begin{proof}
(1$\Rightarrow$2) \ 
Assume $M\models\{\varphi\}'$ and let $N$ be a p-$\omega$-satutated p-elementary superstructure of $M$.
Then $N\models\varphi$ hence $M\models\varphi$ follows by p-maximality.

(2$\Rightarrow$1) \ 
For the converse, assume $N\models\varphi$ where $N$ is some p-elementary extension of $M$.
Let $\varphi'>\varphi''>\varphi$.
We prove $M\models\varphi'$, then $M\models\varphi$ follows from (2).
By Lemma~\ref{lem_interpolation} there is some $\tilde\varphi\perp\varphi''$ such that $\varphi''\rightarrow\neg\tilde\varphi\rightarrow\varphi'$.
Then $\neg\tilde\varphi$ and therefore  $M\models\neg\tilde\varphi$.
Then $M\models\varphi'$.
\end{proof}

\def\ceq#1#2#3{\parbox[t]{20ex}{$\displaystyle #1$}\parbox[t]{6ex}{\hfil $#2$}{$\displaystyle #3$}}

\begin{remark}\label{remk_approx_EH_free}
  By Corollary~\ref{corol_st}, when $\varphi(x\,;\eta)$ does not contain existential quantifiers of sort ${\sf H}$, Proposition~\ref{prop_approx} does not require any assumption of saturation.
  In general, some saturation is necessary: consider the model presented after Corollary~\ref{corol_st} and the formula $\exists x\,\big(x>0\ \wedge\ \iota(x)\in\{0\}\big)$.
  For a counterexample in ${\EuScript F}^{\rm c}$, see Section 5 in~\cite{HI}.
\end{remark}

  
\section{The monster model}\label{monster}

\def\ceq#1#2#3{\parbox[t]{18ex}{$\displaystyle #1$}\parbox{6ex}{\hfil $#2$}{$\displaystyle #3$}}

We denote by \emph{${\EuScript U}$\/} some large p-saturated standard structure which we call the positive monster model.
Truth is evaluated in ${\EuScript U}$ unless otherwise is specified.
We denote by $T$ the positive theory of ${\EuScript U}$.
The density of a set of formulas is understood modulo $T$.
Below we say \emph{p/c-model\/} for p/c-elementary substructure of ${\EuScript U}$.
We stress that the truth of some $\varphi\in{\EuScript F}^{\rm p/c}(M)$ in a p/c-model $M$ implies the truth of $\varphi$ (in ${\EuScript U}$) but not vice versa.
However, all p/c-models agree on the approximated truth.

\begin{fact}
  The following are equivalent 
  \begin{itemize}
    \item[1.] $M$ is a p/c-model;
    \item[2.] $\varphi\ \ \Leftrightarrow\ \ M\models\big\{\varphi\big\}'$\quad for every $\varphi\in{\EuScript F}^{\rm p/c}(M)$.
  \end{itemize}
\end{fact}

\begin{proof}
  (1$\Rightarrow$2) By the same argument as in the proof of Corollary~\ref{corol_pmax}. 
  
  (2$\Rightarrow$1) Clear.
\end{proof}

The following fact demonstrates how positive compactness applies.
There are some subtle differences from the classical setting.
Let $A\subseteq{\EuScript U}$ be a small set throughout this section.

\begin{fact}\label{fact_compactness_imp}
  Let $p(x)\subseteq{\EuScript F}^{\rm p}(A)$ be a type.
  Then for every  $\varphi(x)\in{\EuScript F}^{\rm p}({\EuScript U})$
  \begin{itemize}
    \item[i.] if $p(x)\rightarrow\neg\varphi(x)$ then for some conjunction $\psi(x)$ of formulas in $p(x)$
    
    \noindent\kern-\leftmargin
    \ceq{\hfill\psi(x)}{\rightarrow}{\neg\varphi(x);}

    \item[ii.] if $p(x)\rightarrow\varphi(x)$ and $\varphi'>\varphi$ then for some conjunction $\psi(x)$ of formulas in $p(x)$
    
    \noindent\kern-\leftmargin
    \ceq{\hfill\psi(x)}{\rightarrow}{\varphi'(x).}
  \end{itemize} 
\end{fact}

\begin{proof}
  (i) is immediate by saturation; (ii) follows from (i) by Lemma~\ref{lem_interpolation}.
\end{proof}

\begin{fact}\label{fact_Fdense}
  Let ${\EuScript H}$ be a p/c-dense set of positive/continuous formulas.
  Then ${\EuScript H}'$ is p/c-dense.
\end{fact}

\begin{proof}
  Let $\varphi'>\varphi$.
  Pick $\varphi''$ such that $\varphi'>\varphi''>\varphi$.
  By density, $\varphi(x)\rightarrow\psi(x)\rightarrow\varphi''(x)$ for som $\psi(x)\in{\EuScript H}$.
  It suffices to prove that $\psi'(x)\rightarrow\varphi'(x)$ for some $\psi'>\psi$.
  By Proposition~\ref{prop_approx}, $\psi(x)\leftrightarrow\{\psi(x)\}'$.
  Therefore, $\psi'(x)\rightarrow\varphi'(x)$ follows from Fact~\ref{fact_compactness_imp}.
\end{proof}

We write \emph{$\mbox{p/c-tp}(a/A)$\/} for the positive/continuous type of $a$ over $A$, that is, the set of formulas $\big\{\varphi(x)\in{\EuScript F}^{\rm p/c}(A)\ :\ \varphi(a)\big\}$.
In general, if ${\EuScript H}$ is any set of formulas, we write \emph{${\EuScript H}\mbox{-tp}(a/A)$\/} for the type $\big\{\varphi(x)\in{\EuScript H}(A)\ :\ \varphi(a)\big\}$. 
The undecorated symbol \emph{$\mbox{tp}(a/A)$\/} denotes the ${\EuScript L}$-type.

\begin{fact}\label{fact_max_cons_L}
  Let $p(x)\subseteq{\EuScript F}^{\rm p/c}(A)$.
  The following are equivalent
  \begin{itemize}
    \item[1.] $p(x)$ is a subset of ${\EuScript F}^{\rm p/c}_x(A)$ that is maximally consistent in ${\EuScript U}^{|x|}$;
    \item[2.] $p(x)=\mbox{p/c-tp}(a/A)$ for some $a\in{\EuScript U}^{|x|}$.
  \end{itemize}
\end{fact}

\begin{proof}
  (1$\Rightarrow$2) \ 
  If $p(x)$ is consistent then $p(x)\subseteq\mbox{p/c-tp}(a/A)$  for some $a\in{\EuScript U}^{|x|}$.
  By maximality $p(x)=\mbox{p/c-tp}(a/A)$.

  (2$\Rightarrow$1) \ 
  Suppose $\varphi(x)\in{\EuScript F}^{\rm p/c}(A)\smallsetminus  p$.
  Then $\neg\varphi(a)$.
  Hence $\tilde\varphi(a)$ holds for some $\tilde\varphi\perp\varphi$.
  Therefore $p(x)\rightarrow\neg\varphi(x)$ follows.
\end{proof}

Fact~\ref{fact_max_cons_L} extends as follows.

\begin{fact}\label{fact_max_cons_F}
  Let $p(x)\subseteq{\EuScript H}(A)$, where ${\EuScript H}$ is a p/c-dense set of positive/continuous formulas.
  Then the following are equivalent 
  \begin{itemize}
    \item[1.] $p(x)$ is a maximally consistent subset of ${\EuScript H}_x(A)$;
    \item[2.] $p(x)={\EuScript H}\mbox{-tp}(a/A)$ for some $a\in{\EuScript U}^{|x|}$;
    \item[3.] $p(x)\leftrightarrow\mbox{p/c-tp}(a/A)$ for some $a\in{\EuScript U}^{|x|}$.
  \end{itemize}

\end{fact}
  
\begin{proof}
  (1$\Rightarrow$2) \ 
  As above.
  
  (2$\Rightarrow$3) \ 
  As $q(x)\leftrightarrow q'(x)$, then $q(x)\leftrightarrow p(x)={\EuScript H}\mbox{-tp}(a/A)$.

  (3$\Rightarrow$1) \ 
  Immediate.
\end{proof}

\begin{proposition}\label{prop_Fapprox}
  Let ${\EuScript H}$ be a p/c-dense set of positive/continuous formulas.
  Then for every formula $\varphi(x)\in{\EuScript F}^{\rm p/c}$
  
  \ceq{{\rm i.}\hfill\neg\varphi(x)}
  {\leftrightarrow}
  {\bigvee\big\{\psi(x)\in{\EuScript H}\ :\  \psi'(x)\rightarrow\neg\varphi(x)\textrm{ for some }\psi'>\psi\big\};}

  \ceq{{\rm ii.}\hfill\neg\varphi(x)}
  {\leftrightarrow}
  {\bigvee\big\{\neg\psi(x)\ :\ \neg\psi(x)\rightarrow\neg\varphi(x)\textrm{ and }\psi(x)\in{\EuScript H}\big\}.}

\end{proposition}
 
\begin{proof} 
  (i) \ Only $\rightarrow$ requires a proof.
  Let $a\in {\EuScript U}^{|x|}$ be such that $\neg\varphi(a)$.
  Let $p(x)=\mbox{p/c-tp}(a)$.
  By Fact~\ref{fact_max_cons_L}, $p(x)\rightarrow\neg\varphi(x)$.
  As ${\EuScript H}$ is p/c-dense, $p(x)\leftrightarrow q(x)={\EuScript H}\mbox{-tp}(a)$.
  Then $q'(x)\rightarrow\neg\varphi(x)$ hence, by compactness, $\psi'(x)\rightarrow\neg\varphi(x)$ for some $\psi(x)\in{\EuScript H}$.

  (ii) \ By density

  \ceq{\hfill\varphi(x)}
  {\leftrightarrow}
  {\bigwedge\big\{\psi(x)\in{\EuScript H}\ :\ \varphi(x)\rightarrow\psi(x)\big\}.}
  
  Negating both sides of the equivalence we obtain (ii).
\end{proof}

\section{Cauchy's completions}\label{Cauchy}

We recall a few definitions from~\cite{clcl}. For $\tau(x,z)=\tau_1(x,z),\dots,\tau_n(x,z)$ a tuple of terms of sort ${\sf H}^{|x|+|z|}\to {\sf S}$ we define the formula

\ceq{\hfill\emph{$x\sim_\tau y$}}{=}{\bigwedge_{i=1}^n\forall z\ \ \tau_i(x,z)=\tau_i(y,z),}

where the expression $\alpha=\beta$ is shorthand for $\langle\alpha,\beta\rangle\in\Delta$, where $\Delta$ is the diagonal of $S^2$.
We also define the type 

\ceq{\hfill\emph{$x\sim y$}}{=}{\Big\{x\sim_\tau y\ :\ \tau(x,z)\textrm{ as above}\Big\}.}

The following fact will be used below without mention.
It is proved by induction on $\lambda$.



\begin{fact}\label{fact_productUniformity}
  For any $a=\langle a_i:i<\lambda\rangle$ and $b=\langle b_i:i<\lambda\rangle$ 
  
  \ceq{\hfill a\sim b}{\Leftrightarrow}{a_i\sim b_i}\quad for every $i<\lambda$.
\end{fact}

  









Note that the approximations of the formula $x\sim_\tau\!y$ have the form

\ceq{\hfill\emph{$x\sim_{\tau,D} y$}}{=}{\bigwedge_{i=1}^n\forall z\ \ \langle \tau_i(x,z),\tau_i(y,z)\rangle\in D}

for some compact neighborhood $D$ of $\Delta$.
We write \emph{${\EuScript E}$} for the set containing the pairs $\tau,D$ as above.
The formulas $x\sim_\varepsilon y$, as $\varepsilon$ ranges over ${\EuScript E}$, form a base for a system of entougages on ${\EuScript U}^{|x|}$.
We refer to this uniformity and the topology associated to it as the uniformity, respectively topology, \emph{induced by $S$.}
Though not needed in the sequel, it is worth mentioning that the topology induced by $S$ on ${\EuScript U}^{|x|}$ coincides with the product of the topology induced by $S$ on ${\EuScript U}$.
See~\cite{clcl} for details.

The following is easily proved by induction on the syntax.

\begin{fact}\label{fact_continuous}
  For every $\varphi(x)\in{\EuScript F}^{\rm c}$ 

  \ceq{\hfill x\sim y}{\rightarrow}{\varphi(x)\leftrightarrow\varphi(y).}
\end{fact}

\begin{proof}
  %
  %
  %
  %
  %
  %
  %
\end{proof}

We define

\ceq{\hfill\emph{$x\sim'y$}}{=}{\Big\{x\sim_\varepsilon y\ :\ \varepsilon\in{\EuScript E}\Big\}.}

The following corollary corresponds to the Perturbation Lemma~\cite{HI}*{Proposition~5.15}.

\begin{corollary}\label{corol_pertubation}
  For every $\varphi(x)\in{\EuScript F}^{\rm c}$ and every $\varphi'>\varphi$ there an $\varepsilon\in{\EuScript E}$ such that

  \ceq{\hfill x\sim_\varepsilon \!y\ \wedge\ \varphi(y)}{\rightarrow}{\varphi'(x)}
\end{corollary}

\begin{proof}
  As $x\sim' y\ \cup\ \big\{\varphi(x)\big\}\ \rightarrow\ \varphi(y)$ by Fact~\ref{fact_continuous}, the corollary follows from Fact~\ref{fact_compactness_imp}.
\end{proof}

We say what a type $q(x)$ is \emph{finitely satisfiable\/} in $A$ if every conjunction of formulas in $q(x)$ has a solution in $A^{|x|}$.
This definition coincides with the classical one, but in our context, the notion is less robust.
For instance, if $M$ is a p-model and $q(x)=\mbox{p-tp}(a/M)$ then $q'(x)$ is always finitely satisfiable while $q(x)$ need not.

A set $A\subseteq{\EuScript U}$ is \emph{Cauchy complete\/} if it contains all those $a\in{\EuScript U}$ such that $a\sim'x$ is finitely satisfied in $A$.
Note that Cauchy complete sets are in particular closed under the equivalence $(\sim)$.
The \emph{Cauchy completion\/} of $A$ is the set 

\ceq{\hfill\emph{${\rm Ccl}(A)$}}{=}{\big\{a:a\sim'x\textrm{ is finitely satisfied in }A\big\}}.

\begin{fact}\label{fact_finsat}
  ${\rm Ccl}(A)$ is Cauchy complete.
\end{fact}

\begin{proof}
  Suppose that $x\sim'a$ is finitely satisfied in ${\rm Ccl}(A)$.
  Let $\varepsilon\in{\EuScript E}$ be given.
  We prove that  $a\sim_\varepsilon x$ is satisfied in $A$.
  Let $\eta\in{\EuScript E}$ be such that $x\sim_\eta y\sim_\eta z\ \rightarrow\ x\sim_\varepsilon z$.
  There is some $b\in{\rm Ccl}(A)$ such that $b\sim_\eta a$.
  There is some $c\in A$ such that $c\sim_\eta b$.
  Then $c$ satisfies $x\sim_\varepsilon a$, as required.
\end{proof}

We say that $p(x)\subseteq{\EuScript F}^{\rm c}({\EuScript U})$ is a \emph{Cauchy type\/} if it is consistent and $p(x)\wedge p(y)\rightarrow x\sim y$.

\begin{fact}
  Let $M$ be a p/c-model and let $p(x)\subseteq{\EuScript F}^{\rm p/c}(M)$ be a Cauchy type.
  Then all realizations of $p(x)$ belong to ${\rm Ccl}(M)$.
\end{fact}

\begin{proof}
  If $p(x)$ is a Cauchy type, then $p(x)\rightarrow a\sim x$ for some $a\models p(x)$.
  As $M$ is a c-model, $p'(x)$ is finitely satisfied in $M$. 
  Then also $a\sim' x$ is finitely satisfied. 
  Hence $a\in{\rm Ccl}(M)$.
\end{proof}


\section{Morphims}\label{morphisms}

The notion of p-elementarity is discussed in detail in~\cite{clcl}.
In this section we are mostly concerned with its continuous analogue.
This is delicate due to the lack of equality of sort ${\sf H}$.
A possible approach is to work in the quotient ${\EuScript U}/{\sim}$.
A second possibility is to replace functions with relations.
The two options differs only in the notation.
Here we go for the second one that, arguably, simplifies the bookkeeping.

For simplicity, we only discuss morphisms between saturated structures.
We fix, besides ${\EuScript U}$, a second saturated structure ${\EuScript V}$.

Let $R\mathrel{\subseteq}{\EuScript U}\times{\EuScript V}$ be a binary relation.
If $a=\langle a_i:i<\lambda\rangle$ and $b=\langle b_i:i<\lambda\rangle$ we write $\langle a,b\rangle\in R$ to abbreviate: $\langle a_i,b_i\rangle\in R$ for all $i<\lambda$.
Let ${\EuScript H}$ be a set of formulas.
We say that $R$ preserves the truth of ${\EuScript H}$-formulas if 

\ceq{\#\hfill{\EuScript U}\models\varphi(a)}{\Rightarrow}{{\EuScript V}\models\varphi(b)}\hfill for every $\varphi(x)\in{\EuScript H}$ and every $\langle a,b\rangle\in R$.

We will consider two extreme cases for ${\EuScript H}$: the set of all positive/continuous formulas and the set of atomic positive/continuous formulas.
We denote the latter by \emph{${\EuScript F}^{\rm p/c}_{\rm at}$.}
Relations that preserve the truth of ${\EuScript F}^{\rm p/c}$-formulas are called \emph{p/c-elementary.}
Note that p-elementary relations are (graphs of) injective functions and most facts below become trivial.
Therefore we only discuss the continuous case.


By the following fact, if $R$ preserves the truth of ${\EuScript H}$-formulas so does $R^{-1}$.

\begin{fact}\label{fact_Rinverse}
  Let ${\EuScript H}$ be either ${\EuScript F}^{\rm c}$ or ${\EuScript F}^{\rm c}_{\rm at}$.
  Let $R$ be a relation that preserves the truth of ${\EuScript H}$-formulas.
  Then the converse implication in (\#) holds.
\end{fact}

\begin{proof}
  ${\EuScript H}={\EuScript F}^{\rm c}$. \ 
  Assume ${\EuScript U}\models\neg\varphi(a)$ and apply Proposition~\ref{prop_Fapprox} to infer that ${\EuScript U}\models\psi(a)$ for some $\psi(x)\in{\EuScript F}^{\rm c}$ that implies $\neg\varphi(x)$.
  Then ${\EuScript V}\models\neg\varphi(b)$ follows.

  ${\EuScript H}={\EuScript F}^{\rm c}_{\rm at}$. \ 
  Assume ${\EuScript U}\models\tau(a)\notin C$ then ${\EuScript U}\models\tau(a)\in\tilde{C}$ for some compact $\tilde{C}$ disjoint of $C$.
  Then ${\EuScript V}\models\tau(b)\notin C$ follows.
\end{proof}

%
%
%

We denote by $(\sim)$ the equivalence relation defined as in Section~\ref{Cauchy}; we use the same symbol for the relation in ${\EuScript U}$ and in ${\EuScript V}$.
The context will disambiguate.
We say that the relation $R$ is \emph{reduced\/} if there is no $R'\ \subset\ R$ such that $R\ \subseteq\ \mathrel{(\sim)}\circ \mathrel{R'}\circ\mathrel{(\sim)}$.

\begin{fact}
  For every $R$ there is a reduced relation $R'\subseteq R$ such that 
  
  \ceq{\hfill\mathrel{(\sim)}\circ \mathrel{R}\circ\mathrel{(\sim)}}{=}{\mathrel{(\sim)}\circ \mathrel{R'}\circ\mathrel{(\sim)}.}
\end{fact}

\begin{proof}
  Let $\langle R_i:i<\lambda\rangle$ be a decreasing chain of maximal length of relations such that $R_i\mathrel{\subseteq}R\ \subseteq\ \mathrel{(\sim)}\circ \mathrel{R_i}\circ\mathrel{(\sim)}$.
  Let  $R'$ the intersection  of the chain.
  By maximality, $R'$ is a reduced relation.
  Finally, $\mathrel{(\sim)}\circ \mathrel{R}\circ\mathrel{(\sim)}\ \ =\ \mathrel{(\sim)}\circ \mathrel{R'}\circ\mathrel{(\sim)}$ is clear.
\end{proof}




\begin{fact}\label{fact_reduced_funct} 
  Let ${\EuScript H}={\EuScript F}^{\rm c}$.
  Let $R$ be a reduced relation that preserves of the truth of ${\EuScript H}$-formulas.
  Then $R$ is (the graph of) an injective map $f:{\rm dom}(R)\to{\rm range}(R)$.

  The same holds when ${\EuScript H}={\EuScript F}^{\rm c}_{\rm at}$ if ${\rm dom}(R)$ and ${\rm range}(R)$ are c-models.
\end{fact}

\begin{proof}
  ${\EuScript H}={\EuScript F}^{\rm c}$. \ 
  Let $a\in{\rm dom}(R)$ and assume that $\langle a,b\rangle,\langle a,b'\rangle\in R$.
  To prove functionality suppose for a contradiction that $b\neq b'$.
  As $R$ is reduced, $b\not\sim b'$.
  Then $b\not\sim_\varepsilon b'$ for some $\varepsilon\in{\EuScript E}$.
  By Proposition~\ref{prop_Fapprox} $b,b'$ satisfies a continuous formula $\psi(x,x')$ that implies  $x\not\sim_\varepsilon x'$.
  Then $a,a$ also satisfies $\psi(x,x')$, a contradiction.
  This proves functionality, for injectivity apply the same argument to $R^{-1}$.

  ${\EuScript H}={\EuScript F}^{\rm c}_{\rm at}$. \ 
  Reason as above to obtain $b\not\sim_\varepsilon b'$, say $\varepsilon=\langle\tau,D\rangle$.
  Then the formula 
  
  \ceq{\hfill\neg\varphi(b,b',z)}{=}{\neg\bigwedge_{i=1}^n\langle\tau_i(b,z),\tau_i(b',z)\rangle\in D}.
  
  is consistent.
  Then it has a solution, say $c$, in ${\rm range}(R)^{|z|}$.
  Then $\neg\varphi(a,a,d)$ for some $d\mathrel{R} c$ which is a contradiction.
\end{proof}

Let $f \subseteq\ R$ be a reduced c-elementary relation such that $R\ \subseteq\ \mathrel{(\sim)}\circ \mathrel{f}\circ\mathrel{(\sim)}$.
Then $f$ is a function defined in exactly one representative $a$ of each $(\sim)$-equivalence class $[a]$ that intersects ${\rm dom}(R)$.

\begin{fact}\label{fact_Rcompletion}
  Let ${\EuScript H}$ be either ${\EuScript F}^{\rm c}$ or ${\EuScript F}^{\rm c}_{\rm at}$.
  The following are equivalent
  \begin{itemize}
    \item[1.] $R$ preserves the truth of ${\EuScript H}$-formulas;
    \item[2.] ${\rm Ccl}(R)$ preserves the truth of ${\EuScript H}$-formulas.
  \end{itemize}
\end{fact}

\begin{proof}
  We prove the nontrivial implication, 1$\Rightarrow$2.
  Let $\varphi(x)\in{\EuScript H}$ and $\varphi''>\varphi'>\varphi$ be given.
  Let $\langle a,b\rangle\in{\rm Ccl}(R)$ and assume $\varphi(a)$.
  Let $\varepsilon\in{\EuScript E}$ be such that 

  \ceq{1.\hfill x\sim_\varepsilon y\wedge\varphi\phantom{'}(x)}{\rightarrow}{\varphi'(y)}

  \ceq{2.\hfill x\sim_\varepsilon y\wedge\varphi'(x)}{\rightarrow}{\varphi''(y)}

  By the definition of Cauchy completion, there are $a',b'\sim_\varepsilon a,b$ such that $\langle a',b'\rangle\in R$.
  Then from (1) we infer that $\varphi'(a')$.
  Note that $\varphi'(x)\in{\EuScript H}$, then $\varphi'(b')$. 
  Then $\varphi''(b)$ follows from (2).
  As $\varphi''>\varphi$ is arbitrary, the fact follows from Proposition~\ref{prop_approx}.
\end{proof}

\begin{definition}\label{def_ciso}
  Let $M$ and $N$ be p/c-elementary substructures of ${\EuScript U}$, respectively ${\EuScript V}$.
  A \emph{p/c-isomorphism\/} between $M$ and $M$ is a relation $R$ that preserves the truth of ${\EuScript F}^{\rm p/c}_{\rm at}$-formulas and $M={\rm dom}(R)$,  $N={\rm range}(R)$.
\end{definition}

 It is clear that p-isomorphisms coincide with ${\EuScript L}$-isomorphisms. 
 Indeed, in this case $R$ is a bijective function that preserve the truth of atomic and negated atomic formulas (by the same argument as in Fact~\ref{fact_Rinverse}).
 On the other hand, in the continuous case the term isomorphism may sound less appropriate.
 It is justified by the following fact.

\begin{fact}\label{fact_c_iso}
  Let $R$ be a c-isomorphism between $M$ and $N$.
  Let $f\subseteq{\rm Ccl}(R)$ be a reduced relation such that ${\rm Ccl}(R)\ =\ \mathrel{(\sim)}\circ \mathrel{f}\circ\mathrel{(\sim)}$.
  Then $f:{\rm Ccl}(M)\rightarrow{\rm Ccl}(N)$ is an injective map such that 
  
  \ceq{\hfill{\EuScript U}\models\varphi(a)}{\Leftrightarrow}{{\EuScript V}\models\varphi(fa)}\hfill for all $\varphi(x)\in{\EuScript F}^{\rm c}$ and $a\in{\rm dom}(f)^{|x|}$.
\end{fact}

\begin{proof}
  From Fact~\ref{fact_reduced_funct} we obtain that $f:{\rm Ccl}(M)\rightarrow{\rm Ccl}(N)$ is an injective map.
  To prove $\varphi(a)\leftrightarrow\varphi(fa)$ we reason by induction as in the classical case.
\end{proof}

Note that, though $f:{\rm Ccl}(M)\rightarrow{\rm Ccl}(N)$ above is neither total nor surjective, ${\rm dom}(f)$ and ${\rm range}(f)$ intersect every $(\sim)$-equivalence class.

\section{Elimination of quantifiers of sort \textsf{S}}\label{cIelimination}

We write \emph{${\EuScript F}^{\rm p/c}_{{\sf S}{\rm qf}}$} for the set of positive/continuous formulas without quantifiers of sort ${\sf S}$.
In~\cite{clcl} it is proved that ${\EuScript F}^{\rm p}_{{\sf S}{\rm qf}}$ is a p-dense set.
Here we adapt the argument to the continuous case.

\begin{proposition}\label{prop_Iqf_elim}
  Every relation of small cardinality $R$ that preserves the truth of formulas in ${\EuScript F}^{\rm p/c}_{{\sf S}{\rm qf}}$ extends to p/c-automorphisms of ${\EuScript U}$.
\end{proposition}

\begin{proof}
  We prove the proposition in the continuous case.
  A simplified version of the same argument applies to the positive case.

  We construct by back-and-forth a sequence $f_i:{\EuScript U}\rightarrow{\EuScript U}$ of reduced relations (i.e.\@ injective maps) that preserve the truth formulas in ${\EuScript F}^{\rm c}_{{\sf S}{\rm qf}}$.
  Finally we set 
  
  \ceq{\hfill R}{=}{\bigcup_{i<|{\EuScript U}|}(\sim)\circ f_i\circ(\sim).}

  We will ensure that ${\EuScript U}={\rm dom}(R)={\rm range}(R)$, hence that $R$ is a c-automorphism of ${\EuScript U}$.

  Let $i$ be even.
  Let $a$ be an enumeration of ${\rm dom}(f_n)$.
  Let $b\in{\EuScript U}$. 
  Let $p(x,y)=\mbox{c-tp}(a,b)$.
  By c-elementarity the type $p(f_ia,y)$ is finitely satisfied.
  Let $c\in{\EuScript U}$ realize $p(f_ia,y)$.
  Let $f_{i+1}$ be a reduced relation such that $f_{i+1}\ \subseteq\ f_i\cup\{\langle b,c\rangle\}\ \subseteq\ (\sim)\circ f_{i+1}\circ(\sim)$.

  When $i$ is odd we proceed similarly with $f_i^{-1}$ for $f_i$.
  Limit stages are obvious.
\end{proof}

\begin{corollary}\label{corol_cLcomplete}
  Let $p(x)$ and $q(x)$ be as in (i) or (ii) below.
  Then $p(x)\leftrightarrow q(x)$.
  \begin{itemize}
    \item[i.] $p(x)={\EuScript F}^{\rm p}_{{\sf S}{\rm qf}}\mbox{-tp}(a)$ and $q(x)=\mbox{tp}(a)$.
    \item[ii.] $p(x)={\EuScript F}^{\rm c}_{{\sf S}{\rm qf}}\mbox{-tp}(a)$ and $q(x)=\mbox{c-tp}(a)$.
  \end{itemize}
\end{corollary}

\begin{proof}
  (i) \ 
  By Proposition~\ref{prop_Iqf_elim} as p-isomorphisms are ${\EuScript L}$-isomorphisms.

  (ii) \ 
  Only $\rightarrow$ requires a proof.
  If $a$ contains entries $a_i\sim a_j$, replace $a_j$ with $a_i$.
  Note that the tuple $a'$ obtained in this manner has the same c-type of $a$.
 
  Let $b\models p(x)$ and let $b'$ obtained with the same procedure as $a'$.
  
  Then the function that maps $a'\mapsto b'$ preserves the truth of ${\EuScript F}^{\rm c}_{{\sf S}{\rm qf}}$-formulas.
  Then $f$ extends to an c-automorphism.
  As every c-automorphism is c-elementary the corollary follows.
\end{proof}

\begin{proposition}\label{prop_cLHapprox1}
  The set ${\EuScript F}^{\rm p/c}_{{\sf S}{\rm qf}}$ is p/c-dense modulo $T$.
\end{proposition}

\begin{proof}
  Let $\varphi(x)$ be a positive formula.
  We need to prove that for every $\varphi'>\varphi$ there is some formula $\psi(x)\in{\EuScript F}^{\rm p/c}_{{\sf S}{\rm qf}}$ such that $\varphi(x)\rightarrow\psi(x)\rightarrow\varphi'(x)$.
  By Corollary~\ref{corol_cLcomplete} and Proposition~\ref{prop_approx}

  \ceq{\hfill\neg\varphi(x)}{\rightarrow}{\bigvee_{p'(x)\rightarrow\neg\varphi(x)}p'(x)}

  where $p(x)$ ranges over the maximally consistent ${\EuScript F}^{\rm p/c}_{{\sf S}{\rm qf}}$-types.
  By Fact~\ref{fact_compactness_imp} and Lemma~\ref{lem_interpolation}

  \ceq{\hfill\neg\varphi(x)}{\rightarrow}{\bigvee_{\neg\tilde{\psi}(x)\rightarrow\neg\varphi(x)}\neg\tilde{\psi}(x),}

  where $\tilde{\psi}(x)\in{\EuScript F}^{\rm p/c}_{{\sf S}{\rm qf}}$.
  Equivalently,

  \ceq{\hfill\varphi(x)}{\leftarrow}{\bigwedge_{\tilde{\psi}(x)\leftarrow\varphi(x)}\tilde{\psi}(x).}

  By compactness, see Fact~\ref{fact_compactness_imp}, for every $\varphi'>\varphi$ there are some finitely many $\tilde{\psi}_i(x)\in{\EuScript F}^{\rm p}_{{\sf S}{\rm qf}}$ such that

  \ceq{\hfill\varphi'(x)}{\leftarrow}{\bigwedge_{i=1,\dots,n}\tilde{\psi}_i(x)\ \ \leftarrow\ \ \varphi(x)}

  which yields the interpolant required by the proposition.
\end{proof}

\section{The Tarski-Vaught test and the L\"owenheim-Skolem theorem}

The following proposition is our version of the Tarski-Vaught test.

\begin{theorem}\label{thm_Tarski_Vaught}
  Let $M$ be a subset of ${\EuScript U}$.
  Let ${\EuScript H}$ be a p-dense set of positive formulas.
  Then the following are equivalent
  \begin{itemize}
    \item[1.] $M$ is a p-model;
    \item[2.] for every formula $\psi(x)\in{\EuScript H}(M)$
    
    \noindent\kern-\leftmargin
    \ceq{\hfill\exists x\,\psi(x)}{\Rightarrow}
    {\textrm{ for every }\psi'>\psi\textrm{ there is an }a\in M\textrm{ such that }\psi'(a);}
  
    \item[3.] for every formula $\psi(x)\in{\EuScript H}(M)$
    
    \noindent\kern-\leftmargin
    \ceq{\hfill \exists x\,\neg\psi(x)}{\Rightarrow}
    {\textrm{ there is an }a\in M\textrm{ such that }\neg\psi(a).}
  \end{itemize}
  If ${\EuScript H}$ is a c-dense set of continuous formulas then (2) and (3) are equivalent to
  \begin{itemize}
    \item[1$'$.] ${\rm Ccl}(M)$ is a c-model.
  \end{itemize}
  Moreover, if $M$ is a substructure, then (1$'$), (2) and (3) are equivalent to
  \begin{itemize}
    \item[1$''$.] $M$ is a c-model.
  \end{itemize}
\end{theorem}

\begin{proof}
  (1$\Rightarrow$2) \ 
  Assume $\exists x\,\psi(x)$ and let $\psi'>\psi$ be given.
  By Lemma~\ref{lem_interpolation} there is some $\tilde{\psi}\perp\psi$ such that  $\psi(x)\rightarrow\neg\tilde{\psi}(x)\rightarrow\psi'(x)$.
  Then $\neg\forall x\,\tilde{\psi}(x)$ hence, by (1), $M\models\neg\forall x\,\tilde{\psi}(x)$.
  Then $M\models\neg\tilde{\psi}(a)$ for some $a\in M$. Hence $M\models\psi'(a)$ and $\psi'(a)$ follows from (1).

  (2$\Rightarrow$3) \ 
  Assume (2) and let $\psi(x)\in{\EuScript H}(M)$ be such that $\exists x\,\neg\psi(x)$.
  By Proposition~\ref{prop_Fapprox}.i, there are a consistent $\varphi(x)\in{\EuScript H}(M)$ and some $\varphi'>\varphi$ such that $\varphi'(x)\rightarrow\neg\psi(x)$.
  Then (3) follows.

  (3$\Rightarrow$2) \ 
  Let $\psi'>\psi$ for some $\psi(x)\in{\EuScript H}(M)$.
  Let $\tilde{\psi}\bot\psi$ such that $\psi(x)\rightarrow\neg\tilde{\psi}(x)\rightarrow\psi'(x)$.
  By Proposition~\ref{prop_Fapprox}.ii, $\neg\varphi(x)\rightarrow\neg\tilde{\psi}(x)$ for some $\varphi(x)\in{\EuScript H}(M)$ such that $\neg\varphi(x)$ is consistent.
  Then (2) follows from (3).
  
  (2$\Rightarrow$1) \ 
  Assume (2).
  By the classical Tarski-Vaught test $M\preceq_{\sf H}{\EuScript U}$.
  Then $M$ is the domain of a substructure of ${\EuScript U}$.
  Then $M\models\varphi(a)\ \Rightarrow\ \varphi(a)$ holds for every atomic formula $\varphi(x)$ and for every $a\in M^{|x|}$.
  Now, assume inductively
  
  \ceq{\hfill M\models\varphi(a,b)}{\Rightarrow}{\varphi(a,b).}

  Using (2) and the induction hypothesis we prove that

  \ceq{\hfill M\models\exists y\,\varphi(a,y)}{\Rightarrow}{\forall y\,\varphi(a,y)}.

  Indeed, for any $\varphi'>\varphi$,

  \ceq{\hfill M\models\exists y\,\varphi(a,y)}
  {\Rightarrow}{M\models\exists y\,\psi(a,y)}\hfill for some $\psi\in{\EuScript H}$ such that $\varphi\rightarrow\psi\rightarrow\varphi'$
  
  \ceq{}
  {\Rightarrow}
  {M\models\psi(a,b)}\hfill for some $b\in M$ by (2)
  
  \ceq{}
  {\Rightarrow}
  {M\models\varphi'(a,b)}
  
  \ceq{}
  {\Rightarrow}
  {\varphi'(a,b)}\hfill by induction hypothesis

  \ceq{}
  {\Rightarrow}
  {\exists y\,\varphi'(a,y).}

  As $\varphi'>\varphi$ is arbitrary, $\exists y\,\varphi(a,y)$ follows from Proposition~\ref{prop_approx}.

  Induction for the connectives $\vee$, $\wedge$, $\forall\raisebox{1.1ex}{\scaleto{\sf H}{.8ex}\kern-.2ex}$, $\exists\raisebox{1.1ex}{\scaleto{\sf S}{.8ex}\kern-.2ex}$, and $\forall\raisebox{1.1ex}{\scaleto{\sf S}{.8ex}\kern-.2ex}$ is straightforward.

\hfil***

  (1$'\!\Rightarrow$2) \ 
  Let $\psi'>\psi''>\psi$.
  Reasoning as in the proof of (1$\Rightarrow$2) we obtain that ${\rm Ccl}(M)\models\psi''(a)$ for some $a\in{\rm Ccl}(M)$.
  By Corollary~\ref{corol_pertubation}, $a\sim_\varepsilon x\rightarrow\psi'(x)$ for some $\varepsilon$.
  As $a\sim'x$ is finitely satisfied in $M$, it follows that $\psi'(c)$ for some $c\in M$.

  (2$\Leftrightarrow$3) \ 
  The proof above applies verbatim when ${\EuScript H}$ is a c-dense set of continuous formulas.
  
  (2$\Rightarrow$1$'$) \ 
  Assume (2).
  We claim that ${\rm Ccl}(M)$ is a substructure of ${\EuScript U}$.  
  Let $a\in M^{n}$ and let $f$ be a function symbol of sort ${\sf H}^{n}\to{\sf H}$.
  We prove that $fa\in M$.
  We show that $fa\sim' x$ is finitely satisfied in $M$.
  Consider the formula $fa\sim_\varepsilon  x$ where $\varepsilon$ is the pair $\tau,D$.
  By Lemma~\ref{lem_interpolation}, there is a formula in $\tilde\varphi(x)\in{\EuScript F}^{\rm p}(M)$ such that
  
  \ceq{\hfill fa\sim_\tau x}{\rightarrow}{\neg\tilde\varphi(x)}\parbox{6ex}{\hfil$\rightarrow$}$fa\sim_\varepsilon  x$

  By Fact~\ref{fact_Fdense}.ii there is a consistent formula  $\neg\psi(x)$, for some $\psi(x)\in{\EuScript H}(M)$, that implies $fa\sim_\varepsilon  x$.
  Then, by (3), $fa\sim_\varepsilon  x$ is satisfied in $M$.
  This proves our claim.

  Now, we claim that (2) holds also for every $\psi(x)\in{\EuScript H}\big({\rm Ccl}(M)\big)$.
  Let $\psi(x,z)\in{\EuScript H}(M)$ and $\psi'>\psi$ be given.
  Let $b\in{\rm Ccl}(M)^{|z|}$.
  Suppose that $\exists x\,\psi(x,b)$ and let $\varepsilon$ be such that $z\sim_\varepsilon b\rightarrow\exists x\,\psi''(x,z)$ where $\psi'>\psi''>\psi$.
  By Corollary~\ref{corol_pertubation}, we can also assume that $z\sim_\varepsilon b\wedge\psi''(x,z)\rightarrow\psi'(x,b)$. 
  Let $b'\in M^{|z|}$ be such that $b'\sim_\varepsilon b$.
  By (2) there is an $a\in M^{|x|}$ such that $\psi''(a,b')$.
  Then $\psi'(a,b)$ follows.
  This proves the second claim.

  By the two claims above, the inductive argument in the proof in (2$\Rightarrow$1) applies to prove that ${\rm Ccl}(M)$ is a c-model.

  \hfil***

  (1$''\!\Rightarrow$2) \ 
  By the same argument as in (1$\Rightarrow$2).

  (2$\Rightarrow$1$''$) \ 
  As $M$ is a substructure by assumption, the inductive argument in the proof in (2$\Rightarrow$1) applies.
\end{proof}

    
   

\begin{remark}\label{rem_Tarski_Vaught}
  Theorem~\ref{thm_Tarski_Vaught} shows in particular that for every substructure $M$ the following are equivalent
  \begin{itemize}
    \item[1.] $M$ is a c-model;
    \item[2.] ${\rm Ccl}(M)$ is a c-model.
  \end{itemize}
\end{remark}

Classically, the first application of the Tarski-Vaught test is in the proof of the downward L\"owen\-heim-Skolem Theorem.
Note that by the classical downward L\"owenheim-Skolem Theorem every $A\subseteq{\EuScript U}$ is contained in a standard structure of cardinality $|{\EuScript L}(A)|$.
In this form the L\"owenheim-Skolem Theorem is not very informative.
In fact, the cardinality of ${\EuScript L}$ is eccessively large because of the aboundance of symbols in ${\EuScript L}_{\sf S}$.

We say that ${\EuScript F}^{\rm p/c}$ is \emph{separable\/} if there is a countable p/c-dense set ${\EuScript H}$ of positive/continuous formulas.

\begin{proposition}
  Let ${\EuScript F}^{\rm p/c}$ be separable.
  Let $A$ be a countable set.
  Then there is a countable p/c-model $M$ containing $A$.
\end{proposition}

\begin{proof}
  Let ${\EuScript H}$ be a countable p/c-dense set of positive/continuous formulas.
  As in the classical proof of the L\"owenheim-Skolem Theorem, we construct a countable $M\subseteq{\EuScript U}$ that contains a witness of every consistent formula $\neg\psi(x)$ for $\psi(x)\in{\EuScript H}(M)$.
  In continuous case we also ensure that $M$ is a substructure.
  Then the proposition follows from Theorem~\ref{thm_Tarski_Vaught}.
\end{proof}

The following proposition is proved in~\cite{clcl} for a smaller language ${\EuScript L}$ but the proof can be easily adapted.

\begin{proposition}
  Assume $S$ is a second countable (i.e.\@ the topology has a countable base).
  If ${\EuScript L}\smallsetminus{\EuScript L}_{\sf S}$ has at most countably many symbols then ${\EuScript F}^{\rm p}$ is separable.
\end{proposition}

\section{Continuous omitting types}

\def\ceq#1#2#3{\parbox[t]{25ex}{$\displaystyle #1$}\parbox{6ex}{\hfil $#2$}{$\displaystyle #3$}}

In this section we present a version of the omitting types theorem.
We are interested in obtaining Cauchy complete models omitting a given type.
We only consider the continuous case; the positive case is only mentioned incidentally.

A type $p(x)$ is isolated by $\varphi(x)$, a consistent formula, if $\varphi(x)\rightarrow p(x)$.
If $p(x)$ is isolated by $\neg\varphi(x)$ for some $\varphi(x)\in{\EuScript F}^{\rm p/c}(A)$ we say that it is \emph{p/c-isolated\/} by $A$.
We omit the reference to $A$ when this is clear from the context (e.g., when $p(x)$ is presented as a type over $A$).

For ease of language in this and the next section we say that $p(x)$ is realized in $M$ if $p(a)$ holds (in ${\EuScript U}$) for some $a\in M^{|x|}$.
By Theorem~\ref{thm_Tarski_Vaught}, $p(x)$ is p/c-isolated by $A$, then $p(x)$ is realized in every p/c-model containing $A$.

\begin{fact}\label{fact_isolation}
  Let ${\EuScript H}$ be a p/c-dense set of positive/continuous formulas.
  Then the following are equivalent
  \begin{itemize}
  \item[1.] $p(x)$ is p/c-isolated by $A$;
  \item[2.] $p(x)$ is isolated by $\neg\varphi(x)$ for some $\varphi(x)\in{\EuScript H}(A)$;
  \item[3.] $p(x)$ is isolated by some $\varphi'(x)$ such that  $\varphi'>\varphi$ for some consistent $\varphi(x)\in{\EuScript H}(A)$.
  \end{itemize}
\end{fact}

\begin{proof}
  (1$\Rightarrow$2) By Proposition~\ref{prop_Fapprox}.ii.

  (1$\Rightarrow$3) By Proposition~\ref{prop_Fapprox}.i.

  (3$\Rightarrow$1) Let $\varphi'>\varphi$ be as in (3). 
  Let $\tilde\varphi\bot\varphi$ be such that $\varphi(x)\rightarrow\neg\tilde\varphi(x)\rightarrow\varphi'(x)$. 
  Then $\neg\tilde\varphi(x)$ isolates $p(x)$.
\end{proof}

The following fact is stated without proof for the sake of comparison.
It is a form of positive omitting type theorem not very distant from the classical one.
Its continuous analogue fails: consider the type $a\sim x$ which is realized in every model containing $a$ but it in general not c-isolated.

\begin{fact}
  Let ${\EuScript F}^{\rm p}$ be separable.
  Let $A\subseteq{\EuScript U}$ be countable.
  Let $p(x)\subseteq{\EuScript F}^{\rm p}(A)$.
  Then the following are equivalent:
  \begin{itemize}
    \item[1.] $p(x)$ is p-isolated;
    \item[2.] $p(x)$ is realized in every p-model containing $A$.
  \end{itemize}
\end{fact}

We need a weaker notion of isolation.
This will also allow to extend the omitting types theorem to Cauchy complete models.

We say that the type $p(x)\subseteq{\EuScript F}^{\rm c}({\EuScript U})$ is \emph{approximately c-isolated\/} by $A$ if for every $\varepsilon\in{\EuScript E}$ there is a formula $\varphi_\varepsilon(x)\in{\EuScript F}^{\rm c}(A)$ such that $\neg\varphi_\varepsilon(x)$ is consistent with $p(x)$ and isolates $\exists x'\sim_\varepsilon x\ p(x')$.
Notice the requirement of consistency with $p(x)$.
This is used in the proof of Lemma~\ref{lem_atomic_iso}.

Assume for the rest of this section that ${\EuScript F}^{\rm c}$ is separable.
Let $\langle\varepsilon_n:n\in\omega\rangle$ be such that for every $\varepsilon\in{\EuScript E}$ there is an $n$ such that $x\sim_{\varepsilon_n}y\ \rightarrow\ x\sim_\varepsilon y$.
We abbreviate $x\sim_{\varepsilon_n}y$ by \emph{$x\sim_ny$.}

\begin{fact}\label{fact_isolationpp}
  Let ${\EuScript F}^{\rm c}$ be separable.
  Let $p(x)\subseteq{\EuScript F}^{\rm c}(A)$ be maximally consistent type that is approximately c-isolated.
  Then this can be witnessed by formulas $\psi_n(x)\in{\EuScript F}^{\rm c}(A)$ such that $\psi_n(x)\rightarrow\psi_{n+1}(x)$.

  Finally, if $a_n\models\neg\psi_n(x)$, the type $p(x)\cup\big\{a_n\sim_nx\ :\ n\in\omega\big\}$ is consistent.
\end{fact}

\begin{proof}
  Let $\neg\varphi_n(x)$ isolate $\exists x'\sim_n x\ p(x')$.
  Write $\psi_n(x)$ for $\varphi_0(x)\vee\dots\vee\varphi_n(x)$.
  Clearly $\neg\psi_n(x)\ \rightarrow\ x\sim_n\! y$, then we only need verify that $\neg\psi_n(x)$ is consistent with $p(x)$.
  But this is the case because by maximality $p(x)\rightarrow\neg\varphi_n(x)$.
  The second claim follows immediately.
\end{proof}

\begin{fact}\label{fact_sim_n}
  Let ${\EuScript F}^{\rm c}$ be separable.
  Let $p(x)\subseteq{\EuScript F}^{\rm c}(A)$ be maximally consistent and approximately c-isolated by $A$.
  Let $M$ be a c-model containing $A$.
  Then $p(x)$ is realized in ${\rm Ccl}(M)$.
\end{fact}

\begin{proof}
  Let $a_n\in M$ be as in Fact~\ref{fact_isolationpp}.
  Then any solution of $p(x)\cup\big\{a_n\sim_nx\ :\ n\in\omega\big\}$ belongs to ${\rm Ccl}(M)$.
\end{proof}

In words, the proof above shows that if $p(x)$ is approximately c-isolated, then in every c-model contains a sequence of elements that converges, in the topology induced by $S$, to a realization of $p(x)$.

\begin{fact}\label{fact_wcisolatio_realization}
  Let ${\EuScript F}^{\rm c}$ be separable.
  Let $M$ be a c-model containing $A$.
  Then every type $p(x)\subseteq{\EuScript F}^{\rm c}(A)$ that is realized in ${\rm Ccl}(M)$ is approximately c-isolated by $M$.
\end{fact}

\begin{proof}
  Let $b\in{\rm Ccl}(M)^{|x|}$ realize $p(x)$.
  Let $\varepsilon\in{\EuScript E}$ be given.
  By Fact~\ref{fact_isolation} it suffices to find a formula $\varphi(x)$ and some $\varphi'>\varphi$ such that $\varphi'(x)\rightarrow\exists x'\sim_\varepsilon x\ p(x')$.

  Let $\eta'\in{\EuScript E}$ be such that $x\sim_{\eta'} y\sim_{\eta'} z\ \rightarrow\ x\sim_\varepsilon z$.
  Pick some $a\in M^{|x|}$ that satisfies $x\sim_{\eta'} b$.
  Then $x\sim_{\eta'}a\ \rightarrow\exists\ x'\sim_\varepsilon x\ p(x')$.
  Finally, note that $(\sim_{\eta'})>(\sim_\eta)$ for some $\eta\in{\EuScript E}$.
\end{proof}

\begin{lemma}\label{lem_kuratowskiUlam_cont}
  Let ${\EuScript F}^{\rm c}(A)$ be separable.
  Assume that $p(x)\subseteq{\EuScript F}^{\rm c}(A)$ is not approximately c-isolated.
  Then every consistent formula $\neg\psi(z)$, with $\psi(z)\in{\EuScript F}^{\rm c}(A)$, has a solution $a$ such that $A,a$ does not approximately c-isolate $p(x)$.
\end{lemma}

\begin{proof}
  %
  Let $\varepsilon\in{\EuScript E}$ witness that $p(x)$ is not approximately c-isolated by $A$. 
  We construct a sequence of ${\EuScript F}^{\rm c}(A)$-formulas $\langle\gamma_i(z):i<\omega\rangle$ such that any realization $a$ of the type $\big\{\gamma_i(z):i<\omega\big\}$ is the required solution of $\neg\psi(z)$, witnessed by the same $\varepsilon$.
  
  Let $\langle\xi_i(x,z):i<\omega\rangle$ enumerate a countable c-dense subset of ${\EuScript F}^{\rm c}_{x,z}(A)$.
  Let $\gamma_0(z)$ be a consistent continuous formula such that $\gamma'_0(z)\rightarrow\neg\psi(z)$ for some $\gamma'_0>\gamma_0$.
  This exists by Proposition~\ref{prop_Fapprox}.
  Now we define $\gamma_{i+1}(z)$ and $\gamma'_{i+1}>\gamma_{i+1}$.
  Let $\gamma_i(z)$ and $\gamma'_i>\gamma_i$ be given.
  Pick $\tilde{\gamma}\perp\gamma_i$ such that $\gamma_i(z)\rightarrow\neg\tilde{\gamma}(z)\rightarrow\gamma'_i(z)$.
  
  \begin{itemize}
  \item[1.] If $\neg\xi_i(x,z)\wedge\gamma_i(z)$ is inconsistent with $p(x)$, let $\gamma_{i+1}(z)=\gamma_i(z)$ and  $\gamma'_{i+1}(z)=\gamma'_i(z)$.
  \item[2.] Otherwise, pick $\varphi(x)\in p$ such that (\#) below is consistent (the existence of such formula is proved below)
  
  (\#)\hfil $\neg\tilde{\gamma}(z)\ \wedge\ \exists x\,\big[\neg\xi_i(x,z)\wedge\neg\exists x'\sim_\varepsilon x\ \varphi(x')\big].$\kern15ex
  
  Finally, let $\gamma_{i+1}(z)$ and $\gamma'_{i+1}>\gamma_{i+1}$ be consistent continuous formulas such that $\gamma'_{i+1}(z)$ implies (\#).
  Such formulas exist by Proposition~\ref{prop_Fapprox}.
  \end{itemize}
  
  Let $a\models\{\gamma_i(z):i<\omega\}$.
  We claim that that $A,a$ does not c-isolate $p(x)$, witnessed by $\varepsilon$.
  Otherwise $\neg\xi_i(x,a)\rightarrow\exists x'\sim_\varepsilon x\  p(x')$ for some $\neg\xi_i(x,a)$ consistent with $p(x)$.
  This contradicts $a\models\gamma_{i+1}(z)$.
  
  Therefore the proof is complete if we can show that it is always possible to find the formula $\varphi(x)$ required in (2).
  
  Suppose for a contradiction that $\neg\xi_i(x,z)\wedge\gamma_i(z)$ is consistent with $p(x)$ while (\#) is inconsistent for all formulas $\varphi(x)\in p$.
  This immediately implies that 
  
  \ceq{\hfill\exists z\;\big[\neg\xi_i(x,z)\wedge\neg\tilde{\gamma}(z)\big]}{\rightarrow}{\exists x'\sim_\varepsilon x\ p(x').}
  
  By Fact~\ref{fact_isolation}, this yields the desired contradiction.
\end{proof}

\begin{theorem}[Continuous Omitting Types]\label{thm_cOTT}
  Let ${\EuScript F}^{\rm c}$ be separable.
  Let $A$ be countable.
  Assume also that $p(x)\subseteq{\EuScript F}^{\rm c}(A)$ is not approximately c-isolated.
  Then there is a c-model $M$ containing $A$ such that ${\rm Ccl}(M)$ omits $p(x)$. 
\end{theorem}

\begin{proof}
  By assumption $\exists y\sim_\varepsilon x\ p(y)$ is not approximately c-isolated for some $\varepsilon\in{\EuScript E}$. 
  Just as in the classical case, we apply Lemma~\ref{lem_kuratowskiUlam_cont} and the Tarski-Vaught test (Theorem~\ref{thm_Tarski_Vaught}) to obtain a countable c-model $M$ that does not isolate $\exists y\sim_\varepsilon x\ p(y)$.
  By Fact~\ref{fact_wcisolatio_realization}, ${\rm Ccl}(M)$ omits $p(x)$. 
\end{proof}

\section{Continuous countable categoricity}

\def\ceq#1#2#3{\parbox[t]{25ex}{$\displaystyle #1$}\parbox[t]{6ex}{$\displaystyle\hfil #2$}{$\displaystyle #3$}}

Let ${\EuScript F}^{\rm c}$ be separable.
We say that $T$ is \emph{c-$\omega$-categorical\/} if any two countable c-models are c-isomorphic as defined in Definition~\ref{def_ciso}.
It is easy to see that c-$\omega$-categorical theories are complete.
This allows us to work inside a monster model and simplify the notation.

Let $M$ be a c-model.
We say that $M$ is \emph{atomic\/} if for every finite tuple $a$ of elements of $M$ the type $\mbox{c-tp}(a)$ is approximately c-isolated.

\begin{lemma}\label{lem_atomic_iso}
  Let ${\EuScript F}^{\rm c}$ be separable.
  Let $M$ and $N$ be two atomic c-models.
  Let $k$ be a finite c-elementary relation between ${\rm Ccl}(M)$ and ${\rm Ccl}(N)$.
  Then for every $c\in M$ there is $d\in{\rm Ccl}(N)$ such that $k\cup\{\langle c,d\rangle\}$ is a c-elementary relation.
\end{lemma}

\begin{proof}
  We can assume that $k$ is reduced, i.e.\@ it is an injective map ${\rm Ccl}(M)\rightarrow{\rm Ccl}(N)$.
  Let $a$ enumerate ${\rm dom}(k)$.
  Let $p(x,y)=\mbox{c-tp}(a,c)$, it suffices to prove that $p(x,y)$ is realized by some $d\in{\rm Ccl}(N)$.
  By Fact~\ref{fact_isolationpp} there are some formulas $\psi_n(x,y)\in{\EuScript F}^{\rm c}$ such that
  
  \ceq{\hfill\neg\psi_n(x,y)}{\rightarrow}{\exists x',y'\sim_nx,y\ \ p(x',y')}
  
  and $\psi_n(x,y)\rightarrow\psi_{n+1}(x,y)$.
  As $\neg\psi_n(x,y)$ is consistent with $p(x,y)$ then $a,c\models\neg\psi_n(x,y)$.
  Then by c-elementarity, $\neg\psi_n(ka,y)$ is consistent.
  Let $d_n\in{\rm Ccl}(M)$ be such that $ka,d_n\models\neg\psi_n(x,y)$.
  Let $d$ be a realization of the type $\{d_n\sim_ny:n\in\omega\}\cup p(ka,y)$.
  As $d\in{\rm Ccl}(N)$, the proof is complete.
\end{proof}

\begin{fact}\label{fact_atomic_iso}
  Let $M$ and $N$ be countable atomic c-models.
  Then ${\rm Ccl}(M)$ and ${\rm Ccl}(N)$ are c-isomorphic.
\end{fact}

\begin{proof}
  We construct by back-and-forth a sequence $f_n:{\rm Ccl}(M)\rightarrow{\rm Ccl}(N)$ of finite c-elementary reduced relations (i.e.\@ injective maps) and let 
  
  \ceq{\hfill R}{=}{\bigcup_{n\in\omega}(\sim)\circ f_n\circ(\sim).}

  We will ensure that $M\subseteq{\rm dom}(R)$ and $N\subseteq{\rm range}(R)$, hence that $R$ is a c-isomorphism between $M$ and $N$.

  Let $n$ be even.
  Let $a$ be an enumeration of ${\rm dom}(f_n)$.
  Let $c\in M$.
  Let $p(x,y)=\mbox{c-tp}(a,b)$.
  As $M$ is atomic, by Lemma~\ref{lem_atomic_iso} there is a $d\in{\rm Ccl}(N)$ such that the relation $f_n\cup\{\langle c,d\rangle\}$ is c-elementary.
  Let $f_{n+1}$ be a reduced relation such that $f_{n+1}\ \subseteq\ f_n\cup\{\langle a,c\rangle\}\ \subseteq\ (\sim)\circ f_{n+1}\circ(\sim)$.

  When $n$ is odd we proceed similarly with $f_n^{-1}$.
\end{proof}

The following analogue of Ryll-Nardzewski's Theorem follows by the classical argument.

\begin{theorem}
  Let ${\EuScript F}^{\rm c}$ be separable.
  Then the following are equivalent
  \begin{itemize}
    \item[1.] $T$ is c-$\omega$-categorical;
    \item[2.] every consistent type $p(x)\subseteq{\EuScript F}^{\rm c}$, for $x$ any finite tuple of variables, is approximately c-isolated.
  \end{itemize}
\end{theorem}

\begin{proof}
  (1$\Rightarrow$2) \ 
  If $p(x)$ is consistent it is realized in some countable c-model $M$.
  If $p(x)$ is not approximately c-isolated, by the continuous omitting type theorem there is a c-model $N$ such that  ${\rm Ccl}(N)$ omits $p(x)$.
  Then $M$ and $N$ are not c-isomorphic, hence $T$ is not c-$\omega$-categorical.

  (2$\Rightarrow$1) \ 
  By (2), every c-model is atomic.
  Then c-$\omega$-categoricity follows from Fact~\ref{fact_atomic_iso}.
\end{proof}

\section{Stable formulas}

\def\ceq#1#2#3{\parbox[t]{24ex}{$\displaystyle #1$}\parbox[t]{6ex}{$\displaystyle\hfil #2$}{$\displaystyle #3$}}

The results in this section are completely classical.
The formulas $\varphi(x\,;z)$ and $\tilde\varphi(x\,;z)$ in Theorem~\ref{thm_stability} are not required to be positive, in fact, saturation is not required in this section. 
We prove a version of a fundamental property of stable formulas.
Namely, externally definable sets are internally definable when the defining formula is stable.
The result is relevant to the subject of this paper when we take $\varphi$ positive and $\tilde\varphi\perp\varphi$.

The following definition is from~\cite{Hr}.
We say that the formulas $\varphi(x\,;z),\tilde\varphi(x\,;z)\in{\EuScript L}$ are \emph{stably separated\/} if for some $n<\omega$ there is no sequence $\langle a_i\,;b_i\,:\,i<n\rangle$ such that 

\ceq{{\rm i.}\hfill i<j}{\Rightarrow}{\varphi(a_i\,;b_j)\ \wedge\ \tilde\varphi(a_j\,;b_i)}\hfill  for every $i,j<n$.

Clearly, two stably separated formulas are mutually inconsistent.
Note that $\varphi,\neg\varphi$ are stably separated exactly when $\varphi$ is stable (in ${\EuScript U}$).

We say that the sets ${\EuScript D},\tilde{\EuScript D}\subseteq{\EuScript U}^{|z|}$ are \emph{finitely separated\/} by $\varphi(x\,;z),\tilde\varphi(x\,;z)$ if for every finite set $B\subseteq{\EuScript U}^{|z|}$ there is an $a\in{\EuScript U}^{|x|}$ such that 

\ceq{{\rm ii.}\hfill B\ \cap\ {\EuScript D}}{\subseteq}{B\ \cap\ \varphi(a\,;{\EuScript U});}

\ceq{{\rm iii.}\hfill B\ \cap\ \tilde{\EuScript D}}{\subseteq}{B\ \cap\ \tilde\varphi(a\,;{\EuScript U}).}

Note that this equivalent to requiring that the type below is finitely consistent

\ceq{\hfill p(x)}{=}{\big\{\varphi(x\,;b)\ :\ b\in{\EuScript D}\big\}\ \cup\ \big\{\tilde\varphi(x\,;b)\ :\ b\in\tilde{\EuScript D}\big\}}

Therefore, when $\varphi,\tilde\varphi$ are positive, it follows that ${\EuScript D},\tilde{\EuScript D}$ are \emph{externally separated.}
That is, there is a p-elementary extension ${}^*\!{\EuScript U}$ of ${\EuScript U}$ and an ${}^*\!\!a\in {}^*\!{\EuScript U}^{|x|}$ such that ${\EuScript D}\subseteq\varphi({}^*\!\!a\,;
^*\!{\EuScript U})$ and $\tilde{\EuScript D}\subseteq\tilde\varphi({}^*\!\!a\,;{}^*\!{\EuScript U})$.

Finally, we say that they are \emph{honestly\/} finitely separated if we can also guarantee that 

\ceq{{\rm iv.}\hfill{\EuScript D}\ \cap\ \tilde\varphi(a\,;{\EuScript U})}{=}{\varnothing;}

\ceq{{\rm v.}\hfill\tilde{\EuScript D}\ \cap\ \varphi(a\,;{\EuScript U})}{=}{\varnothing.}

If only (iv) obtains, we say \emph{half-honestly\/} finitely separated.

\begin{lemma}\label{lem_sability1}
  Let $\varphi(x\,;z),\tilde\varphi(x\,;z)$ be stably separated.
  If ${\EuScript D},\tilde{\EuScript D}$ are finitely separated by $\varphi,\tilde\varphi$ then ${\EuScript D},\tilde{\EuScript D}$ are half-honestly finitely separated by $\psi,\tilde\psi$ where 
  
  \ceq{\hfill\psi(x_0,\dots,x_m\,z)}{=}{\varphi(x_0\,;z)\vee\dots\vee\varphi(x_m\,;z)}
  
  \ceq{\hfill\tilde\psi(x_0,\dots,x_m\,z)}{=}{\tilde\varphi(x_0\,;z)\wedge\dots\wedge\tilde\varphi(x_m\,;z)}
  
\end{lemma}

\begin{proof}
  Let $m$ be such that no sequence $\langle a_i\,;b_i : i\le m\rangle$ satisfies (i$'$).
  Let $B\subseteq{\EuScript U}^{|z|}$ be finite.
  It suffices to prove that there are some $a_0,\dots,a_m$ such that 
  
  \ceq{\hfill B\cap{\EuScript D}}{\subseteq}{\varphi(a_i,\,;{\EuScript U})}\hfill for every $0\le i\le m$;
  
  \ceq{\hfill B\cap\tilde{\EuScript D}}{\subseteq}{\tilde\varphi(a_i\,;{\EuScript U})}\hfill for every $0\le i\le m$;

  \ceq{\hfill\varnothing}{=}{{\EuScript D}\ \cap\ \bigcap_{i=0}^m\tilde\varphi(a_i,\,;{\EuScript U}).}

  We define $a_n$ and $b_n$ by recursion.
  Suppose that $a_0,\dots,a_{n-1}$ and $b_0,\dots,b_{n-1}\in{\EuScript D}$ have been defined.
  We first define $a_n$, then $b_n$. 
  Choose $a_n\in{\EuScript U}^{|x|}$ such that 
  
  \ceq{\hfill\{b_0,\dots,b_{n-1}\}\cup \big(B\cap{\EuScript D}\big)}{\subseteq}{\varphi(a_n,\,;{\EuScript U})};

  \ceq{\hfill B\cap\tilde{\EuScript D}}{\subseteq}{\tilde\varphi(a_n\,;{\EuScript U}).}

  This is possible by (ii) and (iii).
  Now, if possible, choose $b_n$ such that

  \ceq{\hfill b_n}{\in}{{\EuScript D}\cap\bigcap_{i=0}^n\tilde\varphi(a_i,\,;{\EuScript U})}.

  If this is not possible, then set $a_i=a_m$ for $n<i\le m$.
  The required tuple is $a_0,\dots,a_m$.
  We claim that at some $n\le m$ the procedure halts.
  In fact, otherwise we would contradict the choice of $m$ as the sequence $\langle a_i\,;b_i : i\le m\rangle$ satisfies (i) up to reversing the order.
\end{proof}

\begin{lemma}\label{lem_sability2}
  Let $\varphi(x\,;z),\tilde\varphi(x\,;z)$ be stably separated.
  If ${\EuScript D},\tilde{\EuScript D}$ is half-honestly finitely separated by $\varphi,\tilde\varphi$ then for some $m$ they are honestly finitely separated by $\psi,\tilde\psi$ where 
  
  \ceq{\hfill\psi(x_0,\dots,x_m\,z)}{=}{\varphi(x_0\,;z)\wedge\dots\wedge\varphi(x_m\,;z)}
  
  \ceq{\hfill\tilde\psi(x_0,\dots,x_m\,z)}{=}{\tilde\varphi(x_0\,;z)\vee\dots\vee\tilde\varphi(x_m\,;z)}
\end{lemma}

\begin{proof}
  Let $m$ be such that no sequence $\langle a_i\,;b_i : i\le m\rangle$ satisfies (i).
  Let $B\subseteq{\EuScript U}^{|z|}$ be finite.
  It suffices to prove that there are some $a_0,\dots,a_m$ such that 
  
  \ceq{\hfill B\cap{\EuScript D}}{\subseteq}{\varphi(a_i,\,;{\EuScript U})}\hfill for every $0\le i\le m$;
  
  \ceq{\hfill B\cap\tilde{\EuScript D}}{\subseteq}{\tilde\varphi(a_i\,;{\EuScript U})}\hfill for every $0\le i\le m$;

  \ceq{\hfill\varnothing}{=}{{\EuScript D}\ \cap\ \tilde\varphi(a_i,\,;{\EuScript U})}\hfill for every $0\le i\le m$;

  \ceq{\hfill\varnothing}{=}{\tilde{\EuScript D}\ \cap\ \bigcap_{i=0}^m\varphi(a_i,\,;{\EuScript U}).}

  We define $a_n$ and $b_n$ by recursion.
  Suppose that $a_0,\dots,a_{n-1}$ and $b_0,\dots,b_{n-1}\in\tilde{\EuScript D}$ have been defined.
  We first define $a_n$, then $b_n$. 
  Choose $a_n\in{\EuScript U}^{|x|}$ such that 
  
  \ceq{\hfill B\cap{\EuScript D}}{\subseteq}{\varphi(a_n,\,;{\EuScript U})};
  
  \ceq{\hfill\{b_0,\dots,b_{n-1}\}\cup \big( B\cap\tilde{\EuScript D}\big)}{\subseteq}{\tilde\varphi(a_n\,;{\EuScript U}).}

  \ceq{\hfill\varnothing}{=}{{\EuScript D}\ \cap\ \tilde\varphi(a_i,\,;{\EuScript U}).}

  This is possible by (ii), (iii) and (iv).
  Now, if possible, choose $b_n$ such that

  \ceq{\hfill b_n}{\in}{\tilde{\EuScript D}\cap\bigcap_{i=0}^n\varphi(a_i,\,;{\EuScript U})}.

  If this is not possible, then set $a_i=a_m$ for $n<i\le m$.
  The required tuple is $a_0,\dots,a_m$.
  We claim that at some $n\le m$ the procedure above halts.
  In fact, otherwise we would contradict the choice of $m$ as the sequence $\langle a_i\,;b_i : i\le m\rangle$ satisfies (i).
\end{proof}

\begin{theorem}\label{thm_stability}
  Let $\varphi(x\,;z),\tilde\varphi(x\,;z)$ be stably separated.
  If ${\EuScript D},\tilde{\EuScript D}$ is finitely separated by $\varphi,\tilde\varphi$ then for some $m,n$ they are honestly finitely separated by $\psi,\tilde\psi$ where 
  
  \ceq{\hfill\psi(\langle x_{i,j}\rangle_{i,j\le m,n}\,;z)}{=}{\bigwedge_{i=0}^m\bigvee_{j=0}^n\varphi(x_{i,j}\,;z)}
  
  \ceq{\hfill\tilde\psi(\langle x_{i,j}\rangle_{i,j\le m,n}\,;z)}{=}{\bigvee_{i=0}^m\bigwedge_{j=0}^n\tilde\varphi(x_{i,j}\,;z)}
\end{theorem}

\begin{proof}
  A well-known Ramsey's argument proves that the formulas $\psi,\tilde\psi$ in Lemma~\ref{lem_sability1} are stably separated.
  Therefore they meet the assumption of Lemma~\ref{lem_sability2} and we can apply Lemma~\ref{lem_sability1} and~\ref{lem_sability1} in succession.
\end{proof}

Let $\xi$ be a variable of sort ${\sf S}$.
Let $\varphi(x\,;z\,;\xi)$ be a formula such that $\forall x,z\ \exists^{=1}\xi\ \varphi(x\,;z\,;\xi)$.
Write $f(x\,;z)=\xi$ for $\varphi(x\,;z\,;\xi)$.
The following remark links our discussion of stability to the double limits interpretation in~\cite{BY}.

\begin{remark}
  The following are equivalent 
  \begin{itemize}
    \item[1.] $f(x\,;z)\in C$ and $f(x\,;z)\in\tilde C$ are stably separated for every disjoint compat sets $C$ and $\tilde C$.
    \item[2.] for every sequence $\langle a_i\,;b_i\,:\,i<\omega\rangle$ 
    
    \ceq{\hfill\lim_{i\to\infty}\lim_{j\to\infty} f(a_i\,;b_j)}{=}{\lim_{j\to\infty}\lim_{i\to\infty} f(a_i\,;b_j)}
  \end{itemize}
\end{remark}
\newcommand\biburl[1]{\url{#1}}
\BibSpec{arXiv}{%
  +{}{\PrintAuthors}{author}
  +{,}{ \textit}{title}
  +{}{ \parenthesize}{date}
  +{,}{ arXiv:}{eprint}
  +{,}{ } {note}
}

\BibSpec{webpage}{%
  +{}{\PrintAuthors} {author}
  +{,}{ \textit} {title}
  +{,}{ } {portal}
  +{}{ \parenthesize} {date}
  +{,}{ } {doi}
  +{,}{ } {note}
  +{.}{ } {transition}
}


\begin{bibdiv}
\begin{biblist}[]\normalsize

\bib{clcl}{article}{
    label={AAVV},
    author = {Agostini, Claudio},
    author = {Baratella, Stefano},
    author = {Barbina, Silvia},
    author = {Motto Ros, Luca},
    author = {Zambella, Domenico},
    title = {Continuous logic in a classical setting},
    note={Submitted},
    date = {2023},
  }

  \bib{A}{webpage}{
    label={A},
    author={Auslander, Josef},
    title={Topological Dynamics},
    portal={Scholarpedia},
    doi={doi:10.4249/ scholarpedia.3449},
    date={2008},
}

  \bib{BY}{article}{
    label={BY},
    author={Ben Yaacov, Ita\"{\i}},
    title={Model theoretic stability and definability of types, after A.
    Grothendieck},
    journal={Bull. Symb. Log.},
    volume={20},
    date={2014},
    number={4},
    pages={491--496},
 }

\bib{BBHU}{article}{
  label={BBHU},
  author={Ben Yaacov, Ita\"{\i}},
  author={Berenstein, Alexander},
  author={Henson, C. Ward},
  author={Usvyatsov, Alexander},
  title={Model theory for metric structures},
  conference={
      title={Model theory with applications to algebra and analysis. Vol. 2},
  },
  book={
      series={London Math. Soc. Lecture Note Ser.},
      volume={350},
      publisher={Cambridge Univ. Press, Cambridge},
  },
  date={2008},
  pages={315--427},
}



\bib{Hr}{article}{
   label={Hr},
   author={Hrushovski, Ehud},
   title={Stable group theory and approximate subgroups},
   journal={J. Amer. Math. Soc.},
   volume={25},
   date={2012},
   number={1},
   pages={189--243},
}


\bib{HI}{article}{
  label={HI},
  author={Henson, C. Ward},
  author={Iovino, Jos\'{e}},
  title={Ultraproducts in analysis},
  conference={
    title={Analysis and logic},
    address={Mons},
    date={1997},
   },
   book={
      series={London Math. Soc. Lecture Note Ser.},
      volume={262},
      publisher={Cambridge Univ. Press, Cambridge},
   },
   date={2002},
   pages={1--110},
}

\bib{K}{arXiv}{
  label={K},
  author = {Keisler, H. Jerome},
  title = {Model Theory for Real-valued Structures},
  eprint={2005.11851},
  doi = {10.48550/ARXIV.2005.11851},
  url = {https://arxiv.org/abs/2005.11851},
  publisher = {arXiv},
  date = {2020},
}





\end{biblist}
\end{bibdiv}
\end{document}